\documentclass[11pt]{article}

\usepackage{amsmath, amssymb, amsthm, mathrsfs, mathtools}
\usepackage{geometry, graphicx}
\usepackage{color, hyperref, tikz}
\usepackage{enumitem}

\definecolor{darkred}{rgb}{0.7, 0.0, 0.0}

\allowdisplaybreaks

\hypersetup{
    colorlinks=true,
    linkcolor=blue,
    citecolor=red,
    urlcolor=cyan
}

\newtheorem{theorem}{Theorem}[section]
\newtheorem{lemma}[theorem]{Lemma}
\newtheorem{proposition}[theorem]{Proposition}
\newtheorem{corollary}[theorem]{Corollary}
\theoremstyle{definition}
\newtheorem{definition}[theorem]{Definition}
\newtheorem{remark}[theorem]{Remark}
\newtheorem{assumption}[theorem]{Assumption}

\newcommand{\R}{\mathbb{R}}

\newcommand{\om}{\boldsymbol{\omega}}
\newcommand{\xiVec}{\boldsymbol{\xi}}
\newcommand{\uVec}{\boldsymbol{u}}
\newcommand{\bmo}{\mathrm{bmo}} 
\newcommand{\vmo}{\mathrm{vmo}} 
\newcommand{\omVec}{\boldsymbol{\omega}}
\newcommand{\Lop}{\mathcal{L}}
\newcommand{\hx}{\hat{\mathbf x}}
\newcommand{\es}{e_*}
\newcommand{\E}{\mathbb{E}}
\newcommand{\Prob}{\mathbb{P}}
\DeclareMathOperator{\osc}{osc}

\title{\textbf{On Decay of the Local Mean Oscillations of the Vorticity Direction in Critical Navier-Stokes Flows}}
\author{Zoran Gruji\'c \\  [0.5cm]  UAB }

\date{\today}

\begin{document}

\maketitle

\begin{abstract}
We isolate and analyze the geometric PDE governing the evolution of the vorticity direction in the 3D incompressible (unforced) Navier-Stokes equations (NSE), restricted to the case of a critical spatial point singularity where the vorticity magnitude concentrates as $O(|x|^{-2})$, inhabiting the critical Lorentz space $L^{3/2, \infty}$. The PDE consists of the Harmonic Map Heat Flow (HMHF) into the sphere supplemented with the fluid transport, cross-diffusion and tangential strain. The question is whether the NSE mechanics can propagate logarithmic decay of the local mean oscillations of the direction -- the condition $\xiVec \in \bmo_{1/|\log r|}$ -- which (in this setting) was shown in the companion paper to prevent finite time blow-up. A key observation is that the $O(|x|^{-2})$ concentration, factored out of the viscous cross-diffusion, generates an outward radial drift $4\nu\, x/|x|^2$ at the core. Under a suitable condition we prove a barrier lemma for the total drift-diffusion operator and show that the HMHF nonlinearity is harmless for the quantity $\frac12|\xiVec - e|^2$ which is a subsolution on any hemisphere. This yields a transfer theorem: the $\bmo_{1/|\log r|}$ regularity of the direction at the core, uniformly up to the singular time, is controlled by a logarithmic modulus in time of the direction at the inner scale, together with the physical strain. Crucially, the strain enters the direction equation only through its tangential component which vanishes precisely when the direction is an eigenvector of the strain tensor. Logarithmic alignment is then shown to imply logarithmic sub-criticality closing the argument. Since this includes the eigenvector carrying the maximal stretching, the result is in contrast to the classical geometric regularity theory which is built on depleting the full vortex-stretching term and thus confined to configurations of weak stretching. 
\end{abstract}

\tableofcontents

\vspace{0.5cm}
\section{Introduction}

The question of whether the 3D Navier-Stokes equations (NSE) can exhibit spontaneous formation of a singularity (no external force) remains a
fundamental open problem in mathematical physics. 
The pursuit of global regularity versus finite time singularity formation hinges on the competition between nonlinear vortex stretching and viscous diffusion. Central to understanding this competition is the local spatial coherence of the vorticity direction, $\xiVec(x,t) = \om(x,t)/|\om(x,t)|$. Pioneering work by Constantin and Fefferman \cite{Constantin1994, Constantin1993} established that Lipschitz continuity of $\xiVec$ in the regions of intense vorticity suffices to prevent blow-up. Subsequent refinements include \cite{BdVBe02}  where this threshold was weakened to $\frac{1}{2}$-H\"older continuity and \cite{Giga2011} where it was shown  -- in the case of Type I blow-up -- that any modulus of uniform continuity suffices (see also \cite{BdVGiGr16}).

\medskip

This paper constitutes the second part of a two-part geometric-analytic framework. In the first part, our companion work \cite{GrujicProj1}, we established -- in the case of a \emph{critical spatial point singularity} (the vorticity magnitude concentrating as $O(|x|^{-2})$, inhabiting the critical Lorentz space $L^{3/2, \infty}$) -- that if the vorticity direction resides locally in a logarithmically weighted space of bounded mean oscillations $\bmo_{1/|\log r|}$, a functional space that permits highly oscillatory, discontinuous topological defects, the vortex stretching is fundamentally depleted. This logarithmic depletion triggers a cascade of spatial scaling imbalances that ultimately forces the flow to avert finite-time blow-up via the harmonic measure maximum principle.

\medskip

The central objective of the present paper is to investigate whether the geometric mechanics of the 3D Navier-Stokes equations are capable of propagating $\bmo_{1/|\log r|}$ regularity. As in \cite{GrujicProj1}, we focus on the case of the critical spatial point singularity. By framing our geometric bounds within this critical concentration, we isolate the nonlinear PDE governing the evolution of the vorticity direction field $\xiVec$. The resulting equation possesses the structural footprint of a Harmonic Map Heat Flow (HMHF) into the spherical target manifold $\mathbb{S}^2$, supplemented with fluid transport, cross-diffusion and tangential strain.

\medskip

In general, such parabolic systems (including a pure HMHF) are susceptible to finite time `topological bubbling' where a geometric singularity (e.g., a dipole bubble) forms and the spatial gradients concentrate at a critical $O(|x|^{-1})$ rate, trapping the field in unweighted $\bmo$ and preventing any decay of local mean oscillations. Our analysis proceeds in three stages.

\begin{enumerate}[leftmargin=*]
    \item \textbf{Stage I: the radial barrier.} Factoring the singular envelope out of the viscous cross-diffusion term generates an \emph{outward radial drift} $+4\nu \frac{\mathbf{x}}{|x|^2} \cdot \nabla$ at the core. Its stationary radial solutions are $1$ and $|x|^3$, and the fluid's own radial velocity competes with it directly: a structural condition, which we call (H$_k$), is that the inward radial fluid velocity at the core stays below $(3-k)\nu/|x|$ for some $k \in (0,3)$. Under (H$_k$) we prove a barrier lemma for the total drift-diffusion operator (Section \ref{sec:barrier}): $|x|^k$ is a supersolution, so the influence of data on a sphere of radius $R_0$ is screened at the core by $(|x|/R_0)^k$, and the diffusion process attached to the operator reaches the inner scale of the singularity in time $\asymp |x|^2/\nu$, with explicit tails obtained by comparison with Bessel processes. The condition (H$_k$) is not implied by $\uVec \in L^{3,\infty}$ or by divergence-freeness, and we state it as a hypothesis. It has a transparent reading -- the singular point is not carried into by its surroundings faster than the viscous drift pushes out -- and Section \ref{sec:application} shows that it is \emph{implied by approximate symmetry of the vorticity magnitude about the local vortex axis}.

    \item \textbf{Stage II: the PDE and the core history.} The centripetal constraint $\nu |\nabla \xiVec|^2 \xiVec$  -- responsible for topological locking of the local mean oscillations in the case of a pure HMHF -- turns out to require no estimate at all: for any fixed unit vector $e$, the deviation $\theta = \frac12|\xiVec - e|^2$ satisfies $\mathcal{L}\theta = -\nu|\nabla\xiVec|^2 (1-\theta)$ and is therefore a subsolution of the linear drift-diffusion operator on the hemisphere $\xiVec \cdot e \ge 0$. Combining this with the barrier lemma, we prove (Section \ref{sec:skeletal}) that the $\bmo_{1/|\log r|}$ norm of $\xiVec(\cdot,t)$ at the core is controlled, uniformly up to the singular time, by the temporal modulus of the direction at the inner scale over the recent past, plus the tangential strain. In particular, a logarithmic modulus in time of the core direction is transferred to a logarithmic modulus in space. Note that a rotating core, whose direction has an $O(1)$ oscillation on every time scale, produces unweighted $\bmo$ and nothing better, so the temporal hypothesis cannot be dropped. At the same time, the core history condition is forgiving -- for example, in the case of the asymptotically self-similar (or asymptotically discretely self-similar) profiles, any scenario in which the core direction settles at least as fast as the slowest generic nonlinear stability (or instability) rate is in. In contrast to the classical geometric PDE techniques where extrinsic gauge transformations or moving frames are deployed to sub-criticalize the HMHF constraint \cite{Uhlenbeck1982, MeyerRiviere2003, Riviere2007, RiviereStruwe2008, Moser2009, Lamm2010, Hong2017}, the sphere's own convexity together with the intrinsic drift of the fluid's critical concentration does the work here. This provides a solid mathematical footing for the phenomenological statement that the viscous shear stress at the apex of a potential tent-like singularity build up would inevitably tear the fold apart, forcing viscous reconnection and spatial smoothing (cf. \cite{GrujicProj1}).

    \item \textbf{Stage III: alignment.} The tangential strain $F_{tan} = S\xiVec - (\xiVec \cdot S\xiVec)\xiVec$ vanishes precisely when $\xiVec$ is an eigenvector of the strain tensor $S$. The physical hypothesis required by Stage II is therefore a logarithmic direction-eigenvector alignment at the core, securing a logarithmic sub-criticality of the tangential strain, $|F_{tan}| \lesssim |x|^{-2}|\log |x||^{-3}$  (the alignment of vorticity with the intermediate strain eigenvector being a robust feature of turbulent flows \cite{Ashurst1987}). We package this into a single application (Section \ref{sec:application}): a strain-aligned, inflow-controlled core does not produce a singularity. 
    \end{enumerate}

A word on the profile. The critical spatial point singularity is defined (Section \ref{sec:drift}) by an upper bound $\omega \le C_0|x|^{-2}$ on an annulus around the singular point together with a logarithmic gradient bound on the shape factor, and \emph{no lower bound}. This is not a matter of generality. An elementary flux computation (Proposition \ref{prop:flux}) shows that a divergence-free field with $c_0|x|^{-2} \le |\om| \le C_0|x|^{-2}$ on an annulus can keep its direction in a fixed open hemisphere on at most $O(C_0/c_0)$ dyadic scales, and, in averaged form, has mean oscillation bounded below on every ball centred at the singular point -- it is never in $\vmo$ there. Since coherence of the direction is what the paper concludes, a two-sided profile would make the conclusion contradict the hypothesis. Conversely, once the direction is coherent, solenoidality forces the magnitude to be algebraically sub-critical away from the inner scale (Proposition \ref{prop:decay}): the core of such a singularity, if it existed, would be localized at the scale $\ell(t)$ with a power-law tail, which is the sub-criticalization that strong anisotropy produces in the Moffatt--Kimura profile. 

\medskip

And a word on what is being depleted, since this is where the present approach departs from the classical geometric regularity theory. The criteria descending from Constantin--Fefferman \cite{Constantin1994, Constantin1993, BdVBe02, BerselliCordoba2009, Giga2011, BdVGiGr16},
as well as the Beltrami and axisymmetric no-swirl mechanisms, all start from the vorticity equation and aim at controlling the full vortex-stretching term $S\om$: coherence of the direction depletes it through cancellations in the Biot-Savart kernel, Beltrami alignment and the no-swirl geometry deplete it algebraically. In every case the favorable configurations are those in which the stretching itself is weak. In the two-part scheme of \cite{GrujicProj1} and the present paper, the stretching is handled in Part I by the logarithmic decay of the local mean oscillations of the direction, and the direction equation \eqref{eq:direction_pde_raw} sees the strain only through its tangential component $F_{tan} = P_{\xiVec^\perp} S\xiVec$. The longitudinal component $\alpha = \xiVec \cdot S\xiVec$, which is the stretching, drops out of the direction dynamics entirely. Now $F_{tan} = 0$ exactly when $\xiVec$ is an eigenvector of $S$, and this holds for any of the three eigenvectors, including the one corresponding to the largest eigenvalue (the configuration of maximal stretching), as well as the one corresponding to the middle eigenvalue (the configuration observed in computational simulations of turbulence ever since  \cite{Ashurst1987}). In this setting, the very mechanism amplifying the vorticity magnitude (stretching) is ultimately responsible for driving the flow into the dissipation range, preventing a finite time blow-up -- a signature of a \emph{self-defeating singularity}.

\section{Formulation of the PDE}

To execute the first stage of our framework we begin by extracting the evolution equation for the direction field. Let $\om = \nabla \times \uVec$ be the vorticity vector associated with a smooth solution to the 3D Navier-Stokes equations. To isolate the purely geometric evolution of the flow, we decompose the field into its scalar magnitude $\omega = |\om|$ and its unit direction $\xiVec = \om/\omega$, defined wherever $\omega > 0$. Substituting the factorization $\om = \omega \xiVec$ into the vorticity equation $\partial_t \om + (\uVec \cdot \nabla)\om = S\om + \nu \Delta \om$ (where $S$ is the strain tensor) and expanding the viscous diffusion term yields:
\begin{equation}
\Delta (\omega \xiVec) = (\Delta \omega) \xiVec + \omega \Delta \xiVec + 2 (\nabla \omega \cdot \nabla) \xiVec.
\end{equation}

Taking the inner product of the expanded vector equation with $\xiVec$, and utilizing the geometric constraints imposed by the spherical target manifold $|\xiVec|^2 = 1$ ($\xiVec \cdot \partial_t \xiVec = 0$, $\xiVec \cdot \nabla \xiVec = 0$, and $\xiVec \cdot \Delta \xiVec = -|\nabla \xiVec|^2$), we extract the scalar evolution equation for the magnitude:
\begin{equation}\label{eq:mag_exact}
\partial_t \omega + \uVec \cdot \nabla \omega - \nu \Delta \omega + \nu \omega |\nabla \xiVec|^2 = \alpha \omega,
\end{equation}
where $\alpha = \xiVec \cdot S \xiVec$ is the scalar stretching eigenvalue. Subtracting this scalar projection (multiplied by $\xiVec$) from the full vector equation and dividing by $\omega$ isolates the evolution of the unit direction:
\begin{equation}\label{eq:direction_pde_raw}
\begin{split}
\partial_t \xiVec + (\uVec \cdot \nabla)\xiVec &= \underbrace{\left[ S\xiVec - (\xiVec \cdot S\xiVec)\xiVec \right]}_{F_{tan}(\xiVec)} \\
&\quad + \nu \Delta \xiVec + \nu |\nabla \xiVec|^2 \xiVec + 2\nu (\nabla \ln \omega \cdot \nabla)\xiVec.
\end{split}
\end{equation}

Two features of \eqref{eq:direction_pde_raw} will matter. First, the strain tensor enters the direction equation only through its tangential component $F_{tan} = S\xiVec - (\xiVec \cdot S\xiVec)\xiVec = P_{\xiVec^\perp} S\xiVec$; the longitudinal component $\alpha = \xiVec \cdot S\xiVec$, which drives the magnitude through \eqref{eq:mag_exact}, does not act on the direction at all. Since $S$ is symmetric, $F_{tan} = 0$ if and only if $\xiVec$ is an eigenvector of $S$, whichever of the three eigenvalues it carries. Second, equation \eqref{eq:direction_pde_raw} possesses the core structural footprint of the HMHF into the unit sphere $\mathbb{S}^2$. In particular, the nonlinear centripetal constraint $\nu |\nabla \xiVec|^2 \xiVec$ acts to keep the vector field on the manifold. In the absence of mitigating fluid mechanisms, this quadratic gradient term is known to drive finite time singularity formation (e.g., formation of a dipole bubble), concentrating spatial gradients at a singular $1/|x|$ rate and locking the vorticity direction in unweighted $\bmo$. In contrast to the classical geometric PDE techniques utilized to sub-criticalize HMHF-type problems, our stabilization mechanism will rely on the spatial interplay between the target manifold constraints and the intrinsic cross-diffusion of the fluid's singular envelope, $2\nu (\nabla \ln \omega \cdot \nabla)\xiVec$.

\section{Drift Decomposition}
\label{sec:drift}

We start by formalizing the macroscopic profile of the singularity. For $t < T^*$ the solution is smooth, so a critical profile cannot hold literally at the singular point; we therefore introduce an inner scale below which the profile is regularized.

\begin{definition}[Critical spatial point singularity]\label{def:critical_point}
A \emph{critical spatial point singularity} at $(0,T^*)$ with inner scale $\ell(t)$ consists of a positive function $\ell$ on $(t_0,T^*)$, $\ell(t) \to 0$ as $t \uparrow T^*$, a radius $R_0>0$, and a shape factor $\Phi$ such that on the annular region
\begin{equation}\label{eq:annulus}
A = \{(x,t):\ \ell(t) < |x| < R_0,\ t_0 < t < T^*\}
\end{equation}
the vorticity magnitude satisfies
\begin{equation}\label{eq:profile}
\omega(x,t) = \Phi(x,t)\,|x|^{-2}, \qquad 0 < \Phi \le C_0, \qquad |\nabla \ln \Phi| \le C_\Phi\,|x|^{-1},
\end{equation}
and the direction satisfies the critical bound
\begin{equation}\label{eq:critgrad}
|\nabla \xiVec(x,t)| \le \frac{C_1}{\max(|x|,\ell(t))} \qquad \text{for } |x| < R_0,\ t_0<t<T^*.
\end{equation}
Inside the inner ball $|x| \le \ell(t)$ the vorticity is bounded by $C_0\,\ell(t)^{-2}$. The velocity is a strong solution on $(0,T^*)$ with $\omega \in L^\infty\bigl((t_0,T^*); L^{3/2,\infty}(\R^3)\bigr)$.
\end{definition}

\begin{remark}
\begin{enumerate}[leftmargin=*]
\item The profile is critical \emph{from above} only. The shape factor captures spatial variations (angular profiles $\Phi \equiv U(x/|x|)$, oscillatory phases like $\Phi = 2 + \sin(\log(1/|x|))$ encoding discrete self-similarity, or algebraically decaying factors $\Phi = (\ell(t)/|x|)^{\delta}$), but no lower bound is imposed on $\Phi$ on the annulus, and none is used below: the drift decomposition \eqref{eq:v_reg} and the barrier lemma of Section \ref{sec:barrier} see $\Phi$ only through $\nabla \ln \Phi$. A lower bound would in fact be fatal, for a kinematic reason that has nothing to do with the dynamics: by the flux obstruction of Proposition \ref{prop:flux} below, a divergence-free field cannot be two-sidedly critical over many dyadic scales while its direction is coherent, and coherence of the direction is exactly what this paper concludes. Proposition \ref{prop:decay} makes this quantitative: under the hypotheses of Theorem \ref{thm:main} the shape factor is forced to decay algebraically away from the inner scale. No smallness is imposed on $C_0$; the only smallness that will be required of $\Phi$ is $C_\Phi < (3-k)/2$ in the inflow condition (H$_k$) of Section \ref{sec:barrier}. In the oscillatory examples this means the admissible oscillation amplitude scales with $\nu$.
\item The critical gradient bound \eqref{eq:critgrad} is not a regularity gain (no logarithm). It is what the profile delivers through the magnitude equation \eqref{eq:mag_exact}, $\nu|\nabla\xiVec|^2 = \alpha - \partial_t \ln\omega - \uVec\cdot\nabla\ln\omega + \nu\Delta\omega/\omega$, once $|\Delta\Phi| \lesssim |x|^{-2}$, $|\partial_t \ln \Phi| \le C$ and $|S| \lesssim |x|^{-2}$ are available; this is the only place where bounds on $\Delta \Phi$ and $\partial_t \Phi$ would enter, and we prefer to state the consequence. It is used only for balls that do not reach the core (Corollary \ref{cor:bmo}).
\item For Type I scenarios $\ell(t) \sim \sqrt{\nu(T^*-t)}$, but nothing below depends on the form of $\ell$.
\item All hypotheses are to be imposed in a frame $x \mapsto x - X(t)$ attached to the singular point. The vorticity and direction equations are invariant under arbitrary translating frames (a uniform acceleration is a gradient force), with $\uVec$ replaced by $\uVec - \dot X(t)$; this matters only for the inflow condition (H$_k$).
\end{enumerate}
\end{remark}

\subsection*{The flux obstruction}

Before extracting the drift we record the kinematic fact behind the one-sidedness of \eqref{eq:profile}. It concerns divergence-free fields only, not solutions, and it is elementary: vortex lines that stay nearly parallel to a fixed direction $n$ carry a flux through the plane $n^\perp$ which, for a magnitude $\sim|x|^{-2}$, grows logarithmically with the number of dyadic scales (the planar integral of $|x|^{-2}$ diverges at the origin), whereas a divergence-free field cannot carry more than $O(C_0)$ flux through any hemisphere.

\begin{proposition}[flux obstruction]\label{prop:flux}
Let $\omVec=\omega\xiVec$ be a smooth divergence-free field on $\{\rho_1<|x|<\rho_2\}$ with
$c_0|x|^{-2}\le\omega\le C_0|x|^{-2}$ there.
\begin{enumerate}[label=(\roman*),leftmargin=*]
\item\label{it:fluxpt} If $\xiVec\cdot n\ge\gamma>0$ on $\{\rho_1<|x|<\rho_2\}$ for some unit vector $n$, then
$\displaystyle \log\frac{\rho_2}{\rho_1}\le\frac{2C_0}{\gamma\,c_0}$.
\item\label{it:fluxavg} If $\rho_2=\rho$, $\rho_1=\rho/K$ with $K:=e^{8C_0/c_0}$, then for every unit vector $\bar\xiVec$
\[
\frac{1}{|B_\rho|}\int_{B_\rho}|\xiVec-\bar\xiVec|\,dx\ \ge\ \tfrac32\,e^{-24C_0/c_0}
\]
(with $\xiVec$ extended arbitrarily as a unit field to $B_{\rho/K}$). In particular, $\xiVec$ is not in $\vmo$ at the origin, and $\|\xiVec\|_{\bmo_{1/|\log r|}(B_\rho)}\ \ge\ \tfrac34\,e^{-24C_0/c_0}\,|\log\rho|$.
\end{enumerate}
\end{proposition}

\begin{proof}
For a unit vector $n$ let $H_n=\{x\cdot n>0\}$, $\Gamma_n(r):=\int_{\{|x|=r\}\cap H_n}\omVec\cdot\hx\,dS$ and $P_n:=\{x\cdot n=0,\ \rho_1<|x|<\rho_2\}$. The divergence theorem on $H_n\cap\{\rho_1<|x|<\rho_2\}$, whose boundary is the outer hemisphere (normal $\hx$), the inner hemisphere (normal $-\hx$) and $P_n$ (normal $-n$), gives
\begin{equation}\label{eq:fluxid}
\int_{P_n}\omVec\cdot n\,d^2x\ =\ \Gamma_n(\rho_2)-\Gamma_n(\rho_1),\qquad\text{so}\qquad\Big|\int_{P_n}\omVec\cdot n\,d^2x\Big|\le4\pi C_0 ,
\end{equation}
since $|\omVec\cdot\hx|\le C_0r^{-2}$ on a hemisphere of area $2\pi r^2$. The right side is bounded independently of $\rho_2/\rho_1$.

\ref{it:fluxpt} $\int_{P_n}\omVec\cdot n\ge\gamma\int_{P_n}\omega\ge\gamma c_0\int_{\rho_1}^{\rho_2}r^{-2}\,2\pi r\,dr=2\pi\gamma c_0\log(\rho_2/\rho_1)$.

\ref{it:fluxavg} Let $N:=\{n\in\mathbb S^2:\ n\cdot\bar\xiVec\ge\tfrac12\}$, of area $\pi$. For $n\in N$, $\xiVec\cdot n\ge\tfrac12-|\xiVec-\bar\xiVec|$, so by \eqref{eq:fluxid}
\[
4\pi C_0\ \ge\ \frac{c_0}{2}\int_{P_n}|x|^{-2}d^2x-C_0\int_{P_n}|x|^{-2}|\xiVec-\bar\xiVec|\,d^2x
=\pi c_0\log K-C_0\int_{P_n}|x|^{-2}|\xiVec-\bar\xiVec|\,d^2x .
\]
Integrate over $n\in N$ and use the co-area identity $\int_{\mathbb S^2}\int_{n^\perp}g\,d^2x\,dn=2\pi\int_{\R^3}g(x)|x|^{-1}dx$ (both sides equal $4\pi^2$ for $g=\mathbf 1_{B_1}$):
\[
\int_N\int_{P_n}|x|^{-2}|\xiVec-\bar\xiVec|\,d^2x\,dn\le2\pi\int_{\rho/K<|x|<\rho}|x|^{-3}|\xiVec-\bar\xiVec|\,dx\le2\pi\Big(\frac K\rho\Big)^3\int_{B_\rho}|\xiVec-\bar\xiVec|\,dx .
\]
Hence $\pi^2c_0\log K\le4\pi^2C_0+2\pi C_0K^3\rho^{-3}\int_{B_\rho}|\xiVec-\bar\xiVec|$, and with $\log K=8C_0/c_0$ and $|B_\rho|=\frac{4\pi}{3}\rho^3$ this is the claim. For the last statement take $\bar\xiVec$ to be the normalized average of $\xiVec$ over $B_\rho$ (when the mean oscillation is less than $\frac12$ the average has norm at least $\frac12$, and $|\xiVec-\bar\xiVec|\le2|\xiVec-\xiVec_{B_\rho}|$).
\end{proof}

\begin{remark}\label{rem:flux}
The conclusion of Corollary \ref{cor:bmo} is a uniform-in-time bound on $\|\xiVec(\cdot,t)\|_{\bmo_{1/|\log r|}(B_{\rho_1})}$, while $\ell(t)\to0$. By \ref{it:fluxavg} this is incompatible with a lower bound $\Phi\ge c_0$ on the annulus for $t$ close to $T^*$. This is why Definition \ref{def:critical_point} is one-sided; the only lower bound compatible with the framework is the one at the inner scale (for a Type~I singularity, $\omega\sim\nu\ell(t)^{-2}$ on $|x|\le\ell(t)$), and Proposition \ref{prop:decay} shows that the transfer theorem then forces the shape factor to decay algebraically on the annulus. 
\end{remark}

The extraction and factorization of the singular envelope from the cross-diffusion drift $\mathbf{V}_{reg} = 2\nu \nabla \ln \omega$ constitutes the core of our geometric stabilization. Substituting the critical point singularity profile, the expansion of the logarithm decomposes the drift into a principal singular radial component and a shape-factor component:
\begin{equation}\label{eq:v_reg}
\mathbf{V}_{reg} = 2\nu \nabla \ln \left( \Phi(x,t)|x|^{-2} \right) = \underbrace{-4\nu \frac{\mathbf{x}}{|x|^2}}_{\mathbf{V}_{reg}^{rad}} + \underbrace{2\nu \frac{\nabla \Phi(x,t)}{\Phi(x,t)}}_{\mathbf{V}_{reg}^{tan}}.
\end{equation}

We rearrange \eqref{eq:direction_pde_raw} into a standard advection-diffusion framework by shifting the total drift to the left-hand side:
\begin{equation}\label{eq:direction_pde_drift}
\begin{split}
\partial_t \xiVec &+ \left( \uVec - 2\nu \frac{\nabla \Phi}{\Phi} \right) \cdot \nabla \xiVec + \underbrace{4\nu \frac{\mathbf{x}}{|x|^2}}_{\text{Active Repulsion}} \cdot \nabla \xiVec \\
&\quad - \nu \Delta \xiVec = F_{tan}(\xiVec) + \nu |\nabla \xiVec|^2 \xiVec.
\end{split}
\end{equation}

This decomposition separates the transport mechanics into two categories.

\begin{enumerate}
    \item \textbf{The fluid and shape-factor drifts.}
    By the Biot-Savart law and the scaling of the critical profile, the advective velocity $\uVec$ resides in $L^\infty\bigl((t_0, T^*); L^{3, \infty}(\R^3)\bigr)$, and in fact satisfies the pointwise bound $|\uVec| \le C_u |x|^{-1}$ near the core (Lemma \ref{lem:BS}); it is divergence-free. The shape-factor drift $-2\nu \nabla \ln \Phi$ is the gradient of a bounded potential with $|\nabla \ln \Phi| \le C_\Phi|x|^{-1}$. These fields are not neutral in general: only their radial components matter for what follows, and their inward radial parts compete directly with the active drift. Divergence-freeness of $\uVec$ does not help here (the divergence-free critical drift theory of \cite{SSSZ2012} operates the best with bounded cutoffs) -- for a divergence-free drift tested against the singular weight $|x|^\gamma$ the advective term reduces to a term in $\uVec \cdot \hat{\mathbf x}$ of the same order as the coercivity it must be absorbed by -- and the condition (H$_k$) of Section \ref{sec:barrier} is a one-sided bound on the inward radial velocity relative to $\nu$.

    \item \textbf{The active radial drift.}
    The principal geometric drift is an outward, repulsive advection $+4\nu \frac{\mathbf{x}}{|x|^2} \cdot \nabla$. Characteristics governed by this field are driven outward from the singularity; equivalently, the diffusion process attached to the operator is drawn into the core and absorbed there. In radial coordinates the operator acting on radial functions is $\nu(\partial_{rr} - \frac{2}{r}\partial_r)$, a Bessel operator of dimension $-1$; its stationary radial solutions are $1$ and $|x|^3$, and the natural symmetrizing measure is $|x|^{-4}dx$. The exponent $3$ is the sharp form of the screening that the barrier lemma of Section \ref{sec:barrier} extracts under (H$_k$) with any exponent $k<3$.
\end{enumerate}

\begin{figure}[htbp]
    \centering
    \begin{minipage}{0.32\textwidth}
        \centering
        \begin{tikzpicture}[scale=0.8]
            \filldraw[black] (0,0) circle (3pt);
            \foreach \angle in {0, 45, 90, 135, 180, 225, 270, 315} {
                \draw[->, very thick, red!80!black] (\angle:2.2) -- (\angle:0.3);
            }
            \node[below] at (0,-2.5) {(a) Unmitigated Bubbling};
        \end{tikzpicture}
    \end{minipage}\hfill%
    \begin{minipage}{0.32\textwidth}
        \centering
        \begin{tikzpicture}[scale=0.8]
            \filldraw[black] (0,0) circle (3pt);
            \foreach \angle in {22.5, 67.5, 112.5, 157.5, 202.5, 247.5, 292.5, 337.5} {
                \draw[->, ultra thick, blue!80!black] (\angle:0.3) -- (\angle:1.2);
                \draw[->, thick, blue!60!black] (\angle:1.3) -- (\angle:1.8);
                \draw[->, thin, blue!40!black] (\angle:1.9) -- (\angle:2.3);
            }
            \node[below] at (0,-2.5) {(b) Regularizing Drift};
        \end{tikzpicture}
    \end{minipage}\hfill%
    \begin{minipage}{0.32\textwidth}
        \centering
        \begin{tikzpicture}[scale=0.8]
            \filldraw[black] (0,0) circle (3pt);
            \draw[dashed, fill=blue!10, draw=blue!80!black, thick] (0,0) circle (1.0);
            \foreach \angle in {0, 90, 180, 270} {
                \draw[->, very thick, purple!80!black] (\angle:2.4) .. controls (\angle:1.4) and (\angle+45:1.1) .. (\angle+70:1.7);
            }
            \node[below] at (0,-2.5) {(c) Stabilized Safe Zone};
        \end{tikzpicture}
    \end{minipage}
    \vspace{0.2cm}
    \caption{Schematic of the screening of the core. (a) The unmitigated Harmonic Map Heat Flow drives gradients toward finite-time concentration at the origin. (b) The critical profile inherently generates an active, singular outward drift $\mathbf{V}_{reg}^{rad} \sim \frac{\mathbf{x}}{|x|^2}$. (c) Under the inflow condition (H$_k$), $|x|^k$ is a supersolution of the total drift-diffusion operator: data on the outer sphere influence the core only through the factor $(|x|/R_0)^k$, while the core history is transported outward undamped.}
    \label{fig:depletion_schematic}
\end{figure}
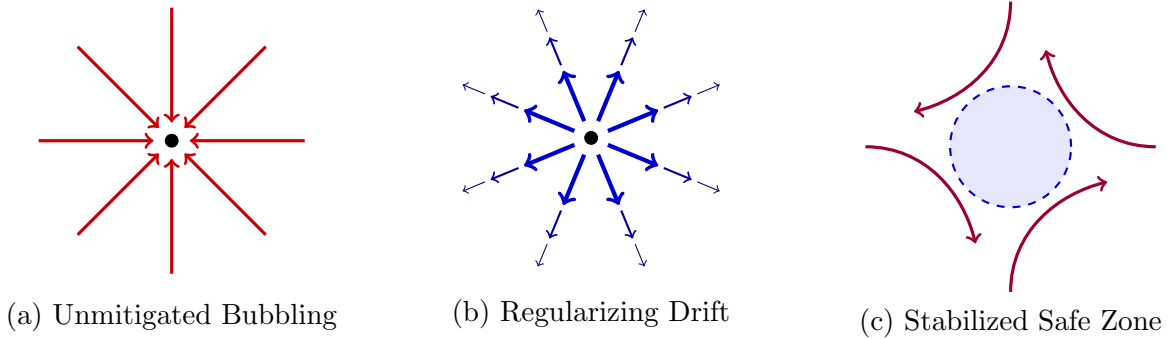

Defining the total drift-diffusion operator
\begin{equation}\label{eq:L_safe}
\Lop = \partial_t - \nu \Delta + b \cdot \nabla, \qquad b = \uVec + 4\nu \frac{\mathbf{x}}{|x|^2} - 2\nu \nabla \ln \Phi ,
\end{equation}
the geometric PDE becomes
\begin{equation}\label{eq:pde_condensed}
\Lop\xiVec = F_{tan}(\xiVec) + \nu |\nabla \xiVec|^2 \xiVec \qquad \text{on } A.
\end{equation}
We keep $-2\nu\nabla\ln\Phi$ as a drift rather than converting it into a zero-order potential by a substitution $\xiVec = \Phi^{-1}\mathbf w$: only its radial component will enter the argument, $\Lop$ then has no zero-order term, constants are exact solutions, and the maximum principle is the classical one.

\section{The Radial Barrier Lemma}
\label{sec:barrier}

This section isolates the linear input that the skeletal PDE (Section \ref{sec:skeletal}) and the application (Section \ref{sec:application}) consume from the operator $\Lop$ of \eqref{eq:L_safe}: a maximum principle, the screening of outer data by the exponent $k$, hitting-time asymptotics of the inner scale, and a barrier for logarithmic sources. All four are comparison statements, provable under one explicit structural hypothesis on the radial component of the drift
(the assumption (H$_k$) below).

\medskip

Write $b_r=b\cdot\hx$, $u_r=\uVec\cdot\hx$, $\hx=x/|x|$, so that
\begin{equation}\label{eq:br}
b_r=\frac{4\nu}{|x|}+u_r-2\nu\,\partial_r\ln\Phi .
\end{equation}
\subsection{Hypotheses}

\begin{assumption}\label{ass:u}
$\nabla\cdot\uVec=0$ and $|\uVec(x,t)|\le C_u|x|^{-1}$ on $A$.
\end{assumption}
The pointwise bound follows from the Biot--Savart law and \eqref{eq:profile} (the Riesz potential of
$|y|^{-2}$ is $|x|^{-1}$ in $\R^3$, convergent at both ends). Divergence-freeness is not used in this section.

\begin{assumption}[radial inflow condition]\label{ass:H}
There is $k\in(0,3)$ such that
\begin{equation}\tag{H$_k$}\label{eq:Hk}
\sup_{A}\ |x|\,\big(u_r\big)_-\;+\;2\nu C_\Phi\;\le\;(3-k)\,\nu .
\end{equation}
\end{assumption}

\begin{remark}\label{rem:H}
\begin{enumerate}[leftmargin=*]
\item \eqref{eq:Hk} is equivalent to $|x|\,b_r\ge(k+1)\nu$ on $A$, i.e.\ the \emph{total} radial drift
$u_r+4\nu/|x|-2\nu\partial_r\ln\Phi$ stays outward with room $(k+1)\nu/|x|$. It bounds only the
\emph{inward} radial component of the fluid velocity relative to viscosity; outward radial flow and
tangential flow are unrestricted.
\item It is not implied by $\uVec\in L^{3,\infty}$ nor by $\nabla\cdot\uVec=0$. In the weighted-energy
formulation the same threshold reappears as the condition under which the advective term
$-\tfrac12\int w^2\,\uVec\cdot\nabla(|x|^\gamma)$ is absorbed by the Hardy sink, so it is the genuine
structural condition and not an artifact of the barrier method.
\item It is sharp for the pure radial model $\uVec\equiv0$, $\Phi\equiv1$: the stationary radial solutions of
$-\nu\Delta+4\nu\frac{x}{|x|^2}\cdot\nabla$ are $1$ and $|x|^3$, and $k\uparrow3$ recovers the exact exponent.
\item The only smallness required of the shape factor of Definition~\ref{def:critical_point} is $2\nu C_\Phi<(3-k)\nu$,
i.e.\ $C_\Phi<\tfrac{3-k}{2}$.
\end{enumerate}
\end{remark}

\subsection{The lemma}

\begin{lemma}[Radial barrier]\label{lem:barrier}
Let $\Lop$ be as in \eqref{eq:L_safe} and let Assumptions~\ref{ass:u} and \ref{ass:H} hold on $A$. Then:
\begin{enumerate}[label=(\roman*),leftmargin=*]
\item\label{it:super} \textup{(Supersolution and maximum principle.)} $\Lop\,1=0$ and
\begin{equation}\label{eq:Lrk}
\Lop\,|x|^k \;=\; k|x|^{k-2}\big(|x|\,b_r-(k+1)\nu\big)\;\ge\;0\qquad\text{on }A.
\end{equation}
Consequently, if $\Lop w\le0$ on $A$ then $\sup_A w\le\sup_{\partial_pA}w$, where $\partial_pA$ is the parabolic
boundary of $A$ (initial slice, inner boundary $|x|=\ell(t)$, outer boundary $|x|=R_0$); in particular
$\|w(t)\|_{L^\infty}$ is non-increasing for solutions.

\item\label{it:screen} \textup{(Screening of outer data.)} If $\Lop w=0$ on $A$ and
$|w|\le A_0\,(|x|/R_0)^k$ on $\partial_pA$ (in particular $w=0$ on the inner boundary and $|w|\le A_0$ on the outer one), then
\begin{equation}\label{eq:screen}
|w(x,t)|\le A_0\Big(\frac{|x|}{R_0}\Big)^k\qquad\text{on }A.
\end{equation}

\item\label{it:hit} \textup{(Hitting time of the inner scale.)} Fix $(x,t)\in A$ and let $X$ solve
$dX_s=-b(X_s,t-s)\,ds+\sqrt{2\nu}\,dW_s$, $X_0=x$, up to the exit time $\tau_A$ from $A$; let
$\tau=\inf\{s:\ |X_s|=\ell(t-s)\}$. Then
\begin{align}
\Prob^x(\tau>s)&\le \frac{1}{\Gamma(\tfrac k2+1)}\Big(\frac{|x|^2}{4\nu s}\Big)^{k/2}\qquad\text{for all }s>0,\label{eq:hit_upper}\\
\Prob^x\big(\tau<\epsilon|x|^2/\nu\big)&\le \eta(\epsilon)\qquad\text{whenever }\ell\le|x|/2,\label{eq:hit_lower}
\end{align}
where $\eta(\epsilon)\to0$ as $\epsilon\to0$ depends only on $C_u/\nu$ and $C_\Phi$. Shortly: $\tau\asymp|x|^2/\nu$.

\item\label{it:source} \textup{(Barrier for logarithmic sources.)} For $\gamma>0$ and $r_*=r_*(\gamma,k)\le e^{-1}$ small,
the function $g(x)=|\log|x||^{-\gamma}$ satisfies, on $A\cap\{|x|<r_*\}$,
\begin{equation}\label{eq:logbarrier}
\Lop g\;\ge\;\frac{\nu\gamma k}{2}\,\frac{1}{|x|^2\,|\log|x||^{\gamma+1}} .
\end{equation}
Consequently, if $\Lop w\le \Lambda\,|x|^{-2}|\log|x||^{-\gamma-1}$ on $A\cap\{|x|<r_*\}$ and $w\le0$ on the
corresponding parabolic boundary, then $w\le \frac{2\Lambda}{\nu\gamma k}\,|\log|x||^{-\gamma}$ there.
\end{enumerate}
\end{lemma}

\subsection{Proof}

\subsubsection*{(i) Supersolution and maximum principle}

In $\R^3$, $\nabla|x|^k=k|x|^{k-2}x$ and $\Delta|x|^k=k(k+1)|x|^{k-2}$. Hence
\[
\Lop|x|^k=-\nu k(k+1)|x|^{k-2}+k|x|^{k-1}\,b_r
= k|x|^{k-2}\big(|x|b_r-(k+1)\nu\big).
\]
By \eqref{eq:br}, $|x|b_r=4\nu+|x|u_r-2\nu\,|x|\partial_r\ln\Phi\ge4\nu-|x|(u_r)_- -2\nu C_\Phi\ge4\nu-(3-k)\nu=(k+1)\nu$,
using \eqref{eq:Hk}. This gives \eqref{eq:Lrk}. Since the coefficients of $\Lop$ are smooth on the closure of $A$
(the region excludes the origin) and there is no zero-order term, the weak maximum principle on the bounded
space--time domain $A$ is classical. \qed

\subsubsection*{(ii) Screening}

Set $\psi=A_0(|x|/R_0)^k$. By (i), $\Lop(\psi\mp w)=\Lop\psi\ge0$ on $A$, and $\psi\mp w\ge0$ on $\partial_pA$
by hypothesis. The maximum principle applied to $-(\psi\mp w)$ gives $\psi\mp w\ge0$ on $A$, i.e.\ \eqref{eq:screen}. \qed

\begin{remark}
Probabilistically, \eqref{eq:screen} with $A_0=1$ is the statement that the process of (iii) exits $A$ through the
outer boundary $|x|=R_0$ before reaching the inner scale with probability at most $(|x|/R_0)^k$: the harmonic
measure of the outer boundary, seen from the core, decays like $|x|^k$. This is the rigorous form of the statement that the origin anchors the direction.
\end{remark}

\subsubsection*{(iii) Hitting time}

Let $\rho_s=|X_s|$. Since $\rho_s\ge\ell>0$ up to $\tau$, It\^o's formula applies with $\nabla|x|=\hx$,
$\Delta|x|=2/|x|$:
\[
d\rho_s=\Big(\frac{2\nu}{\rho_s}-b_r(X_s,t-s)\Big)ds+\sqrt{2\nu}\,dB_s,\qquad dB_s=\hx(X_s)\cdot dW_s,
\]
where $B$ is a one-dimensional Brownian motion (L\'evy's characterization). By \eqref{eq:br} and \eqref{eq:Hk},
\begin{equation}\label{eq:drift_sandwich}
-\frac{(2+C_u/\nu+2C_\Phi)\,\nu}{\rho_s}\;\le\;\frac{2\nu}{\rho_s}-b_r\;\le\;\frac{(1-k)\,\nu}{\rho_s}.
\end{equation}
Let $Y$, $Y'$ solve $dY=\frac{(1-k)\nu}{Y}ds+\sqrt{2\nu}\,dB$ and $dY'=-\frac{(2+C_u/\nu+2C_\Phi)\nu}{Y'}ds+\sqrt{2\nu}\,dB$,
both from $|x|$, driven by the same $B$. With generator $\nu(\partial_{rr}+\frac{d-1}{r}\partial_r)$ these are
time-changed Bessel processes, $Y_s=R^{(d)}_{2\nu s}$ and $Y'_s=R^{(d')}_{2\nu s}$, of dimensions
\begin{equation}\label{eq:dims}
d=2-k\in(-1,2),\qquad d'=-1-\tfrac{C_u}{\nu}-2C_\Phi<0,
\end{equation}
where $R^{(d)}$ denotes the standard Bessel process (generator $\tfrac12(\partial_{rr}+\frac{d-1}{r}\partial_r)$).
On the interval $[0,\tau]$ all three radial processes stay in $[\ell/2,\infty)$ after the obvious localization,
where the drifts are Lipschitz, so the one-dimensional comparison theorem \cite[Thm.~IX.3.7]{RevuzYor1999} (same diffusion coefficient, ordered drifts)
yields
\[
Y'_s\le\rho_s\le Y_s\qquad\text{for }s\le\tau .
\]

\emph{Upper tail.} Since $Y\ge\rho$, $\rho$ reaches $\ell$ no later than $Y$ reaches $\ell$, and a fortiori no later
than $Y$ reaches $0$: $\tau\le\tau^Y_0$. For $d<2$ the hitting time of $0$ by $R^{(d)}$ from $r$ has the law
$r^2/(2\gamma_{(2-d)/2})$, $\gamma_a$ a Gamma$(a,1)$ variable (Getoor \cite{Getoor1979}); undoing the time change,
$\tau^Y_0\overset{d}{=}|x|^2/(4\nu\gamma_{k/2})$. Hence
\[
\Prob^x(\tau>s)\le\Prob\Big(\gamma_{k/2}<\frac{|x|^2}{4\nu s}\Big)\le\frac{1}{\Gamma(\frac k2+1)}\Big(\frac{|x|^2}{4\nu s}\Big)^{k/2},
\]
using $\Prob(\gamma_a<z)\le z^a/\Gamma(a+1)$. This is \eqref{eq:hit_upper}. Note that the exponent $k/2$ is the
one dictated by the scale function $s(r)=r^{2-d}=r^{k}$ of $Y$, consistent with (i)--(ii).

\medskip

\emph{Lower tail.} Since $\rho\ge Y'$ and $\ell\le|x|/2$, $\rho$ cannot reach $\ell$ before $Y'$ reaches $|x|/2$:
$\tau\ge\tau^{Y'}_{|x|/2}$. By Brownian scaling the law of $\nu\,\tau^{Y'}_{|x|/2}/|x|^2$ does not depend on $|x|$,
and it is the law of the time a Bessel process of dimension $d'$ started at $1$ needs to reach $1/2$, which is a.s.\ positive.
Setting $\eta(\epsilon)=\Prob(\tau^{R^{(d')}}_{1/2}<2\epsilon\ ;\ R^{(d')}_0=1)$ gives \eqref{eq:hit_lower}. (Comparison with the
drift-dominated ODE plus Brownian large deviations gives $\eta(\epsilon)\le Ce^{-c/\epsilon}$, but $\eta\to0$ is all that is used.) \qed

\subsubsection*{(iv) Logarithmic source barrier}

Let $g(x)=G(|x|)$ with $G(r)=|\log r|^{-\gamma}=(-\log r)^{-\gamma}$ for $r<1$. Then
$G'(r)=\gamma(-\log r)^{-\gamma-1}r^{-1}>0$ and
\[
G''(r)=\frac{\gamma(\gamma+1)(-\log r)^{-\gamma-2}-\gamma(-\log r)^{-\gamma-1}}{r^2}.
\]
For a radial function with $G'\ge0$, \eqref{eq:br} and \eqref{eq:Hk} give $b\cdot\nabla g=b_rG'\ge\frac{(k+1)\nu}{r}G'$, hence
\[
\Lop g\;\ge\;-\nu\Big(G''+\frac{2}{r}G'\Big)+\frac{(k+1)\nu}{r}G'
=\nu\Big(-G''+\frac{k-1}{r}G'\Big)
=\frac{\nu\gamma}{r^2}\Big(k\,(-\log r)^{-\gamma-1}-(\gamma+1)(-\log r)^{-\gamma-2}\Big).
\]
For $-\log r\ge2(\gamma+1)/k$, the bracket is at least $\tfrac k2(-\log r)^{-\gamma-1}$, which is \eqref{eq:logbarrier}
with $r_*=\exp(-2(\gamma+1)/k)$. The comparison statement follows from (i) applied to
$w-\frac{2\Lambda}{\nu\gamma k}g$. \qed

\begin{remark}
(iv) is the rigorous version of the Green-function heuristic $G(x,y)\lesssim|y|^{-1}$ for $|y|<|x|$: integrating a source
$|y|^{-2}|\log|y||^{-2\beta}$ against $|y|^{-1}$ over $|y|<|x|$ gives $|\log|x||^{1-2\beta}$, which is \eqref{eq:logbarrier}
with $\gamma=2\beta-1$. No information on the transition density is needed; only \eqref{eq:Hk}.
\end{remark}

\section{The Skeletal PDE: Log-Weighted $\bmo$ from the Core History}
\label{sec:skeletal}

We now turn to the geometric PDE \eqref{eq:pde_condensed}. The result of this section is as follows: the log-weighted $\bmo$ of $\xiVec(\cdot,t)$ at the core is controlled by the temporal modulus of the direction at the inner scale over the recent past, plus the tangential strain. The HMHF nonlinearity $\nu|\nabla\xiVec|^2\xiVec$ is absorbed by a convexity identity -- the squared distance to a fixed point of the sphere is a subsolution of the linear part of the flow on the hemisphere around that point -- and the rest is Lemma \ref{lem:barrier}. 

\subsection{Hypotheses}

\begin{assumption}[profile, drift]\label{ass:PH}
On $A$: $\omega=\Phi|x|^{-2}$ with $0<\Phi\le C_0$, $|\nabla\ln\Phi|\le C_\Phi|x|^{-1}$;
$\nabla\cdot\uVec=0$, $|\uVec|\le C_u|x|^{-1}$; and the radial inflow condition \textup{(H$_k$)} for some
$k\in(0,3)$. (These are the hypotheses of Lemma~\ref{lem:barrier}.)
\end{assumption}

\begin{assumption}[critical a priori bound on the direction]\label{ass:grad}
$\displaystyle |\nabla\xiVec(x,t)|\le\frac{C_1}{\max(|x|,\ell(t))}$ for $|x|<R_0$, $t_0<t<T^*$.
\end{assumption}

\begin{remark}
Assumption~\ref{ass:grad} is the critical bound \eqref{eq:critgrad} of Definition~\ref{def:critical_point}, restated for
convenience; it carries no logarithm and is used only for balls that do not reach the core (Corollary~\ref{cor:bmo}).
\end{remark}

\begin{assumption}[cap condition on the boundary data]\label{ass:cap}
There are $e_*\in\mathbb S^2$ and $\phi_0\in(0,\pi/4)$ such that
\[
\xiVec(x,t)\cdot e_*\ \ge\ \cos\phi_0
\qquad\text{for }(x,t)\in\Sigma:=\{|x|\le\ell(t)\}\cup\{|x|=R_0\}\cup\{t=t_0,\ |x|\le R_0\}.
\]
\end{assumption}
That is: on the inner balls, on the outer sphere and at the initial time, the direction stays within a cap of
half-angle $\phi_0<\pi/4$ around a fixed direction. This is a boundary hypothesis only; nothing is assumed in the
interior $\ell(t)<|x|<R_0$. It is what makes $\theta$ below a subsolution on all of $A$, since any two directions in
the cap make an angle $<\pi/2$. Set $\delta_0:=\cos2\phi_0\in(0,1)$.

\begin{definition}[core modulus]\label{def:coremod}
For $t\in(t_0,T^*)$ and $0<\sigma\le t-t_0$,
\[
\omega_{core}(\sigma;t):=\sup\Big\{|\xiVec(y,s)-\xiVec(y',s')|:\ |y|\le\ell(s),\ |y'|\le\ell(s'),\ s,s'\in[t-\sigma,t]\Big\}.
\]
\end{definition}
This is the oscillation of the direction over the inner balls during the time window $[t-\sigma,t]$: the
``core history''.

\begin{assumption}[tangential strain]\label{ass:F}
$\Lop\xiVec=\nu|\nabla\xiVec|^2\xiVec+F$ on $A$ with $|F(x,t)|\le\Lambda\,|x|^{-2}|\log|x||^{-3}$, and $R_0\le r_*(2,k)=e^{-6/k}$
(the threshold of Lemma~\ref{lem:barrier}(iv) with $\gamma=2$). The skeletal PDE is $\Lambda=0$.
\end{assumption}

\subsection{The theorem}

\begin{theorem}[log-$\bmo$ at the core from the core history]\label{thm:main}
Let Assumptions~\ref{ass:PH}, \ref{ass:cap}, \ref{ass:F} hold, and let $R_0$ be small enough that
$\frac{\Lambda}{\nu k}|\log R_0|^{-2}<\tfrac12(1-\delta_0)$. Let $t\in(t_0,T^*)$, $0<\sigma\le t-t_0$,
$2\ell(t)\le\rho\le R_0/2$, and let $e=\xiVec(y_t,t)$ for any $|y_t|\le\ell(t)$. Then
\begin{equation}\label{eq:main}
\sup_{|x|\le\rho}|\xiVec(x,t)-e|\ \le\ \omega_{core}(\sigma;t)
\;+\;\sqrt2\Big(\frac{\rho}{R_0}\Big)^{k/2}
\;+\;\sqrt{\frac{2}{\Gamma(\frac k2+1)}}\Big(\frac{\rho^2}{4\nu\sigma}\Big)^{k/4}
\;+\;\sqrt{\frac{2\Lambda}{\nu k}}\ \frac{1}{|\log\rho|}.
\end{equation}
\end{theorem}

\begin{corollary}[uniform log-$\bmo$]\label{cor:bmo}
Let in addition Assumption~\ref{ass:grad} hold, and suppose the core history has a logarithmic modulus:
\begin{equation}\label{eq:coremod}
\omega_{core}(\sigma;t)\le\frac{\Lambda_0}{|\log\sigma|}\qquad\text{for } \, \ell(t)^2/\nu \le \sigma\le\sigma_1,\ t\in(t_0+\sigma_1,T^*).
\end{equation}
Then there are $\rho_1=\rho_1(\nu,k,\sigma_1)$ and $C=C(k,\nu,R_0,\Lambda,\Lambda_0,C_1)$ such that
\[
\sup_{t\in(t_0+\sigma_1,T^*)}\ \|\xiVec(\cdot,t)\|_{\bmo_{1/|\log r|}(B_{\rho_1})}\ \le\ C,
\]
where the $\bmo_{1/|\log r|}$ semi-norm is taken over all balls $B_r(x_0)\subset B_{\rho_1}$, $0<r<\rho_1$.
\end{corollary}

\begin{remark}[reading the statement]\label{rem:reading}
\begin{enumerate}[leftmargin=*]
\item The right-hand side of \eqref{eq:main} has four terms: the recent core history; the influence of the outer
sphere, screened by the barrier exponent $k$; the influence of the core history older than $\sigma$, screened by the
hitting-time tail; and the tangential strain, integrated by the log-source barrier. The balance between the
second and third terms is the only place where $\sigma$ is chosen; with $\sigma=\rho^2|\log\rho|^{4/k}/\nu$ the third
term is $O(|\log\rho|^{-1})$, and this is what turns \eqref{eq:coremod} into the spatial log-modulus.
\item Nothing is assumed about $\nabla\xiVec$ in the interior. The HMHF nonlinearity $\nu|\nabla\xiVec|^2\xiVec$
never needs to be estimated: it has the right sign in the equation for $\tfrac12|\xiVec-e|^2$.
\item The only exponent in the statement is the barrier exponent $k$ of (H$_k$), which appears with the same role in
both screening terms; no information on the transition density of the process is used.
\item The result is a transfer of regularity from time (at the inner scale) to space (on $B_{R_0}$), with the fluid
mechanics of the core entering through \eqref{eq:coremod}. It is not a self-contained propagation from initial
data: \eqref{eq:coremod} is where the physics of the singular profile has to be supplied. The linear model
($\uVec=0$, $\Phi=1$) shows this cannot be avoided: there, by Lemma~\ref{lem:barrier}(iii), $\xiVec(x,t)$ is an average of the core
history over times $\asymp|x|^2/\nu$, so a core history with an $O(1)$ oscillation on every time scale
(a rotating core) produces $\xiVec(\cdot,t)$ with $O(1)$ oscillation on every spatial scale, i.e.\ unweighted $\bmo$.
\end{enumerate}
\end{remark}

\subsection{Proof of the theorem}

\subsubsection*{Step 1: extension and the process}

Extend $b$ to $\{|x|>R_0\}$ by $b=4\nu x/|x|^2$ (any extension satisfying (H$_k$) will do). For fixed $(x,t)$
let $X$ solve $dX_s=-b(X_s,t-s)\,ds+\sqrt{2\nu}\,dW_s$, $X_0=x$, on $\{|y|>\ell(t-s)\}$, and define the stopping times
\[
\tau=\inf\{s:\ |X_s|=\ell(t-s)\},\qquad \tau_{out}=\inf\{s:\ |X_s|=R_0\},\qquad
\zeta=\tau\wedge\tau_{out}\wedge(t-t_0).
\]
$\zeta$ is the exit time from the space--time region $D=\{(y,s):\ \ell(s)<|y|<R_0,\ t_0<s<t\}$ traversed backward in time.

\subsubsection*{Step 2: the subsolution identity}

Fix $e\in\mathbb S^2$ and set $s_e=\xiVec\cdot e$, $\theta=1-s_e=\tfrac12|\xiVec-e|^2\in[0,2]$. Since $\Lop$ is
linear with no zero-order term and $e$ is constant,
\begin{equation}\label{eq:sub}
\Lop\theta=-\,e\cdot\Lop\xiVec=-\nu|\nabla\xiVec|^2\,s_e-F\cdot e=-\nu|\nabla\xiVec|^2(1-\theta)-F\cdot e .
\end{equation}
Hence, on the set $\{\theta\le1\}=\{\xiVec\cdot e\ge0\}$,
\begin{equation}\label{eq:sub2}
\Lop\theta\le|F|\le\Lambda\,|x|^{-2}|\log|x||^{-3}.
\end{equation}
This is the convexity of the hemisphere: the squared distance to a point of $\mathbb S^2$ is a subsolution of the
linear part of a harmonic-map-type flow on the hemisphere around that point, and the drift does not disturb this
because it acts on $\xiVec$ through a first-order operator that commutes with the constant vector $e$.

\subsubsection*{Step 3: the barrier for the source, and the comparison function}

Let $g(x)=|\log|x||^{-2}$ and $c_\Lambda=\frac{2\Lambda}{2\nu k}=\frac{\Lambda}{\nu k}$. By Lemma~\ref{lem:barrier}(iv) with $\gamma=2$,
$\Lop g\ge\nu k\,|x|^{-2}|\log|x||^{-3}$ on $A$ (recall $R_0\le e^{-6/k}$), so by \eqref{eq:sub2}
\begin{equation}\label{eq:sub3}
\Lop(\theta-c_\Lambda g)\le0\qquad\text{on }D\cap\{\theta\le1\}.
\end{equation}
Let $\Psi$ be the solution of $\Lop\Psi=0$ in $D$ with $\Psi=\theta$ on the parabolic boundary $\partial_pD$
(initial slice $s=t_0$, inner boundary $|y|=\ell(s)$, outer sphere $|y|=R_0$). By Lemma~\ref{lem:barrier}(i), $0\le\Psi\le\sup_{\partial_pD}\theta$,
and the Feynman--Kac representation for the backward process of Step~1 yields
\begin{equation}\label{eq:FK}
\Psi(x,t)=\E^x\big[\theta(X_\zeta,\,t-\zeta)\big].
\end{equation}

\subsubsection*{Step 4: choice of $e$ and the continuity argument}

Take $e=\xiVec(y_t,t)$ with $|y_t|\le\ell(t)$. By Assumption~\ref{ass:cap}, $e\cdot e_*\ge\cos\phi_0$ and every value
of $\xiVec$ on $\Sigma$ makes an angle at most $2\phi_0<\pi/2$ with $e$, so
\begin{equation}\label{eq:capbound}
\sup_{\partial_pD}\theta\ \le\ 1-\cos2\phi_0=1-\delta_0 .
\end{equation}
We claim $\theta<1$ on all of $D$. Let $D_{s_1}=D\cap\{s<s_1\}$ and let $s_1^*$ be the supremum of those $s_1$ for which
$\theta<1$ on $D_{s_1}$; $s_1^*>t_0$ by \eqref{eq:capbound} and continuity. On $D_{s_1^*}$, \eqref{eq:sub3} holds, and
$\theta-c_\Lambda g-\Psi\le -c_\Lambda g\le0$ on $\partial_pD_{s_1^*}$; by the maximum principle (Lemma~\ref{lem:barrier}(i))
\begin{equation}\label{eq:comp}
\theta\ \le\ \Psi+c_\Lambda g\ \le\ (1-\delta_0)+\frac{\Lambda}{\nu k}|\log R_0|^{-2}\ <\ 1-\tfrac12\delta_0\qquad\text{on }D_{s_1^*},
\end{equation}
by the smallness condition on $R_0$. So $\theta$ is bounded away from $1$ on $D_{s_1^*}$ and, by continuity, on a
neighbourhood of its top slice; hence $s_1^*=t$ and \eqref{eq:comp} holds on $D$, in particular at $(x,t)$ for every
$\ell(t)<|x|<R_0$.

\subsubsection*{Step 5: splitting the representation}

Fix $\sigma\in(0,t-t_0]$. In \eqref{eq:FK} split according to how the process leaves $D$:
\begin{align*}
\Psi(x,t)&=\E^x\big[\theta(X_\zeta,t-\zeta);\ \zeta=\tau\le\sigma\big]
+\E^x\big[\theta(X_\zeta,t-\zeta);\ \tau_{out}<\tau\big]
+\E^x\big[\theta(X_\zeta,t-\zeta);\ \zeta\ne\tau_{out},\ \zeta>\sigma\big]\\
&=:I_1+I_2+I_3 .
\end{align*}

\smallskip

\emph{$I_1$.} On this event $X_\zeta$ lies on the inner boundary at a time in $[t-\sigma,t]$, and $e$ is a value on the
inner ball at time $t$, so $\theta(X_\zeta,t-\zeta)=\tfrac12|\xiVec(X_\zeta,t-\zeta)-e|^2\le\tfrac12\omega_{core}(\sigma;t)^2$.
Hence $I_1\le\tfrac12\omega_{core}(\sigma;t)^2$.

\medskip

\emph{$I_2$.} By \eqref{eq:capbound} $\theta\le1$ on $\partial_pD$, and by Lemma~\ref{lem:barrier}(ii) with $A_0=1$
(applied to the solution $h$ of $\Lop h=0$ with $h=1$ on the outer sphere, $h=0$ on the inner boundary and at $t_0$,
which is $h(x,t)=\Prob^x(\tau_{out}<\tau\wedge(t-t_0))$), $I_2\le\Prob^x(\tau_{out}<\tau)\le(|x|/R_0)^k$.

\medskip

\emph{$I_3$.} On this event the process has not reached the inner boundary by time $\sigma$, so $\tau>\sigma$, and
$\theta\le1$; by Lemma~\ref{lem:barrier}(iii), $I_3\le\Prob^x(\tau>\sigma)\le\frac{1}{\Gamma(\frac k2+1)}\big(\frac{|x|^2}{4\nu\sigma}\big)^{k/2}$.

\subsubsection*{Step 6: assembly}

By \eqref{eq:comp} and Step~4, for $\ell(t)<|x|\le\rho$,
\[
\tfrac12|\xiVec(x,t)-e|^2=\theta(x,t)\le\tfrac12\omega_{core}(\sigma;t)^2+\Big(\frac{\rho}{R_0}\Big)^k
+\frac{1}{\Gamma(\frac k2+1)}\Big(\frac{\rho^2}{4\nu\sigma}\Big)^{k/2}+\frac{\Lambda}{\nu k}\,\frac{1}{|\log\rho|^2},
\]
using $|x|\le\rho$ in the last three terms ($|\log|x||\ge|\log\rho|$). Multiply by $2$, take square roots and use
$\sqrt{a+b+c+d}\le\sqrt a+\sqrt b+\sqrt c+\sqrt d$. For $|x|\le\ell(t)$ the bound $|\xiVec(x,t)-e|\le\omega_{core}(\sigma;t)$
holds by definition of the core modulus. This is \eqref{eq:main}. \qed

\subsection{Proof of the corollary}

\emph{Balls centered at the core.} Take $\sigma=\sigma_\rho:=\rho^2|\log\rho|^{4/k}/\nu$. There is $\rho_1=\rho_1(\nu,k,\sigma_1)$ such that for
$\rho\le\rho_1$: $\sigma_\rho\le\sigma_1$, $|\log\sigma_\rho|\ge|\log\rho|$ (since $\log\sigma_\rho=2\log\rho+\frac4k\log|\log\rho|-\log\nu$),
and $(\rho/R_0)^{k/2}|\log\rho|\le1$; moreover $\sigma_\rho\ge4\ell(t)^2|\log\rho|^{4/k}/\nu\ge\ell(t)^2/\nu$ for $\rho\ge2\ell(t)$, so \eqref{eq:coremod} applies. 
Then \eqref{eq:main} and \eqref{eq:coremod} give, for $2\ell(t)\le\rho\le\rho_1$
and $t>t_0+\sigma_1$,
\begin{equation}\label{eq:centred}
\osc_{B_\rho(0)}\xiVec(\cdot,t)\le2\sup_{B_\rho}|\xiVec(\cdot,t)-e|\le\frac{2}{|\log\rho|}\Big(\Lambda_0+\sqrt2+\sqrt{\tfrac{2}{\Gamma(\frac k2+1)}}\,4^{-k/4}+\sqrt{\tfrac{2\Lambda}{\nu k}}\Big)=:\frac{C_2}{|\log\rho|}.
\end{equation}
For $\rho<2\ell(t)$, $B_\rho(0)\subset B_{2\ell(t)}(0)$ and $\osc_{B_\rho}\xiVec\le\osc_{B_{2\ell}}\xiVec\le C_2/|\log(2\ell)|$; if moreover
$\rho\ge\ell|\log\ell|^{-2}$ then $|\log\rho|\le|\log\ell|+2\log|\log\ell|\le2|\log\ell|$, so $\osc_{B_\rho}\xiVec\le 2C_2/|\log\rho|$
(for $\ell\le2\ell$ absorbed into constants); if $\rho<\ell|\log\ell|^{-2}$ then by Assumption~\ref{ass:grad}
$\osc_{B_\rho}\xiVec\le2C_1\rho/\ell\le2C_1|\log\ell|^{-2}\le2C_1/|\log\rho|$, using $|\log\rho|\ge|\log\ell|$ and $|\log\ell|\ge1$.

\medskip

\emph{Off-centered balls.} Let $B_r(x_0)\subset B_{\rho_1}$, $r<\rho_1$, and write $d=|x_0|$.
If $d\le r|\log r|^2$, then $B_r(x_0)\subset B_{\rho}(0)$ with $\rho=d+r\le2r|\log r|^2$, and
$|\log\rho|\ge|\log r|-2\log|\log r|-\log2\ge\tfrac12|\log r|$ for $r$ small; by the centered case
$\osc_{B_r(x_0)}\xiVec\le2C_2\cdot2/|\log r|$.
If $d>r|\log r|^2$, then $B_r(x_0)$ does not meet $B_{d/2}(0)$ (as $r<d/2$) and Assumption~\ref{ass:grad} gives
$|\nabla\xiVec|\le2C_1/\max(d,\ell)\le2C_1/d$ on the ball, hence
$\osc_{B_r(x_0)}\xiVec\le4C_1r/d\le4C_1|\log r|^{-2}\le4C_1/|\log r|$.

\medskip

In all cases $\osc_{B_r(x_0)}\xiVec(\cdot,t)\le C/|\log r|$ with $C$ depending only on $C_2$ and $C_1$; since the mean oscillation
is bounded by the oscillation, the $\bmo_{1/|\log r|}$ semi-norm over $B_{\rho_1}$ is at most $C$, uniformly in
$t\in(t_0+\sigma_1,T^*)$. \qed

\begin{remark}[H\"older variant]
The same argument with the barrier $|x|^m$, $0<m<k$, in place of the logarithmic barrier gives a H\"older transfer:
a H\"older modulus in time of the core history, $\omega_{core}(\sigma;t)\lesssim\sigma^{m/2}$, together with a strain
$|F_{tan}|\lesssim|x|^{-2+m'}$ for some $m'>m$, yields $\osc_{B_\rho(0)}\xiVec(\cdot,t)\lesssim\rho^{m}$ up to the screening term
$(\rho/R_0)^{k/2}$. Since $k<3$ can be taken arbitrarily close to $3$ when the inflow is small, the H\"older exponent
available at the core is limited only by the core history and the strain.
\end{remark}

\subsection{Consequence: the shape factor decays algebraically on the annulus}

Theorem \ref{thm:main} controls the direction; through the flux obstruction it also controls the magnitude. The bound \eqref{eq:centred} is a pointwise hemisphere condition, $\xiVec(\cdot,t)\cdot e_t\ge1-\tfrac12(C_2/2|\log\rho|)^2$ on $B_\rho$, and by Proposition \ref{prop:flux} a two-sided critical magnitude is then possible on $O(1)$ dyadic scales only. The quantitative form is the following.

\begin{proposition}[algebraic decay of the shape factor]\label{prop:decay}
Let the hypotheses of Corollary \ref{cor:bmo} hold, and let $\rho_1$ be as there, with in addition $|\log\rho_1|\ge C_2$. Put
\[
M:=e^{(2\pi+1)C_\Phi},\qquad A:=\frac{8M}{7\log2},\qquad N:=\lceil8A^2\rceil,\qquad \delta:=\frac{1}{5A^2} .
\]
Then for all $t\in(t_0+\sigma_1,T^*)$ and $2\ell(t)\le|x|\le\rho_1$,
\begin{equation}\label{eq:decay}
\Phi(x,t)\ \le\ 2^{6}A\,C_0\,\Big[\Big(\frac{\ell(t)}{|x|}\Big)^{\delta}+\Big(\frac{|x|}{\rho_1}\Big)^{\delta}\Big].
\end{equation}
In particular $\omega(x,t)\le2^6AC_0|x|^{-2}(\ell(t)/|x|)^\delta$ for $2\ell(t)\le|x|\le\sqrt{\ell(t)\rho_1}$: away from the inner scale the profile is algebraically sub-critical.
\end{proposition}

\begin{proof}
Fix $t$ and write $\ell=\ell(t)$, $e=e_t$. By \eqref{eq:centred}, $|\xiVec-e|\le\frac{C_2}{2|\log\rho_1|}\le\frac12$ on $2\ell\le|x|\le\rho_1$, so $\xiVec\cdot e\ge\gamma:=\frac78$ there. Index the dyadic shells $A_j:=\{2^{j+1}\ell\le|x|<2^{j+2}\ell\}$, $0\le j\le J$, $J:=\lfloor\log_2(\rho_1/2\ell)\rfloor-1$, and let $\Phi_j^\pm:=\max/\min_{\overline{A_j}}\Phi(\cdot,t)$. Two points of $\overline{A_j}$ are joined inside it by a radial segment and a spherical arc of total length $\le(2\pi+1)\,2^{j+1}\ell$, on which $|\nabla\ln\Phi|\le C_\Phi/(2^{j+1}\ell)$; hence $\Phi_j^+\le M\Phi_j^-$.

\medskip

\emph{Shell-sum inequality.} For $0\le j\le j'\le J$ apply \eqref{eq:fluxid} with $n=e$ on $\{2^{j+1}\ell<|x|<2^{j'+2}\ell\}$: the planar integral is $\ge\gamma\sum_{i=j}^{j'}\Phi_i^-\int_{A_i\cap e^\perp}|x|^{-2}d^2x=2\pi\gamma\log2\sum_{i=j}^{j'}\Phi_i^-$, while the hemisphere fluxes are bounded by $2\pi\Phi_j^+$ and $2\pi\Phi_{j'}^+$ (the caps lie in $\overline{A_j}$ and $\overline{A_{j'}}$). Therefore
\begin{equation}\label{eq:shellsum}
\sum_{i=j}^{j'}\Phi_i^+\ \le\ A\big(\Phi_j^++\Phi_{j'}^+\big),\qquad A=\frac{M}{\gamma\log2} .
\end{equation}
With $j=0$ this gives $\Phi_i^+\le2AC_0$ for every $i$.

\medskip

\emph{Blocks.} Group the shells into consecutive blocks $I_0,I_1,\dots,I_L$ of $N$ shells ($L=\lfloor J/N\rfloor$; the last block may be shorter, which only helps) and let $m_i:=\max_{I_i}\Phi^+$. By \eqref{eq:shellsum} over $I_i$, $N\min_{I_i}\Phi^+\le2Am_i$, so each block contains a weak shell $w_i$ with $\Phi^+_{w_i}\le\frac{2A}{N}m_i$. Every shell of $I_{i+1}$ lies strictly between $w_i$ and $w_{i+2}$, so by \eqref{eq:shellsum} over $[w_i,w_{i+2}]$,
\[
m_{i+1}\ \le\ A\big(\Phi^+_{w_i}+\Phi^+_{w_{i+2}}\big)\ \le\ \frac{2A^2}{N}\,(m_i+m_{i+2})\ \le\ \frac14\,(m_i+m_{i+2}),\qquad 0\le i\le L-2 .
\]

\medskip

\emph{Discrete maximum principle.} Let $\lambda:=2-\sqrt3$, so that $\lambda+\lambda^{-1}=4$, and $v_i:=2AC_0(\lambda^i+\lambda^{L-i})$; then $v_{i+1}=\frac14(v_i+v_{i+2})$ exactly, and $v_0,v_L\ge2AC_0\ge m_0,m_L$. If $d_i:=m_i-v_i$ had a positive maximum at an interior index $i^*$, then $d_{i^*}\le\frac14(d_{i^*-1}+d_{i^*+1})\le\frac12d_{i^*}$, a contradiction; hence $m_i\le2AC_0(\lambda^i+\lambda^{L-i})$ for all $i$.

\medskip

\emph{Back to $x$.} For $x\in A_j$ with $j\in I_i$: $i\ge j/N-1$ and $2^j\ge|x|/(4\ell)$, so $\lambda^i\le\lambda^{-1}(|x|/4\ell)^{-\log_2(1/\lambda)/N}\le\lambda^{-1}4^{\delta}(\ell/|x|)^{\delta}$, using $\log_2(1/\lambda)=\log_2(2+\sqrt3)>1.6$ and $\delta=1/(5A^2)\le1.6/N$; similarly $L-i\ge(J-j)/N-2$ and $2^{J-j}\ge\rho_1/(16|x|)$ give $\lambda^{L-i}\le\lambda^{-2}16^{\delta}(|x|/\rho_1)^{\delta}$. Since $\lambda^{-2}<14$ and $16^\delta\le2$, \eqref{eq:decay} follows with the constant $2\cdot14\cdot2\,AC_0\le2^6AC_0$.
\end{proof}

\begin{remark}\label{rem:decay}
\begin{enumerate}[leftmargin=*]
\item Nothing in the proof uses the dynamics beyond \eqref{eq:centred}: once the direction is coherent on the annulus, solenoidality alone forces the magnitude to be algebraically localized near the inner scale (and near the fixed outer radius $\rho_1$, where the self-similar region ends and the flux can be supplied from outside). 
\item The statement reads, contrapositively, as a description of what a critical spatial point singularity with coherent direction would have to look like: a Type~I core of extent $O(\ell(t))$, two-sidedly critical over at most $O(A)$ dyadic scales, with a tail decaying faster than $|x|^{-2}$ by a power. This is the sub-criticalization that strong anisotropy produces in the Moffatt--Kimura profile \cite{GrujicMK}, here derived from the divergence constraint and the direction coherence alone. It is also the reason the companion paper \cite{GrujicProj1} works with an upper envelope on the super-level sets: a lower bound at the critical level, spread over many scales, is not something a coherent vortex core can carry.
\item For the application in Section \ref{sec:application} this means that the strain-aligned, inflow-controlled cores of Definition \ref{def:core}, if they existed, would be algebraically localized; the conditions of that definition are then to be read on the $O(A)$ scales where the magnitude is critical and on the inner ball, which is where the core history enters.
\end{enumerate}
\end{remark}
\section{The Tangential Strain: Alignment and Exact Depletions}
\label{sec:strain}

At a critical point singularity the Biot-Savart law scales the strain tensor as $|S| \lesssim |x|^{-2}$, so the tangential strain is a priori critical, $|F_{tan}| \lesssim |x|^{-2}$, and Theorem \ref{thm:main} asks for a logarithmic improvement, $|F_{tan}| \le \Lambda |x|^{-2}|\log|x||^{-3}$. This section clarifies what such an improvement means geometrically.

\subsection{Alignment}

Since $S$ is symmetric,
\begin{equation}\label{eq:Ftan}
F_{tan}=S\xiVec-(\xiVec\cdot S\xiVec)\,\xiVec=P_{\xiVec^\perp}S\xiVec,
\qquad S=\tfrac12(\nabla\uVec+\nabla\uVec^{\top}),
\end{equation}
vanishes if and only if $\xiVec$ is an eigenvector of $S$. Writing $S = \sum_i \lambda_i\, e_i \otimes e_i$ with $\lambda_1 \ge \lambda_2 \ge \lambda_3$ and $\xiVec = \sum_i c_i e_i$, one has $|F_{tan}|^2 = \sum_i \lambda_i^2 c_i^2 - \bigl(\sum_i \lambda_i c_i^2\bigr)^2$, which vanishes exactly when $\xiVec$ lies in an eigenspace. In particular, if $\xiVec$ makes an angle $\phi$ with the eigenvector $e_j$, then, since $(S-\lambda_j I)e_j = 0$ and $|\xiVec - (\xiVec\cdot e_j)e_j| = \sin\phi$,
\begin{equation}\label{eq:Ftan_angle}
|F_{tan}| = \big|P_{\xiVec^\perp}(S-\lambda_j I)\xiVec\big| \le \|S-\lambda_j I\|\,\sin\phi \le (\lambda_1-\lambda_3)\sin\phi ,
\end{equation}
and for the extremal eigenvectors ($j=1,3$) the sharp bound is $|F_{tan}| \le \tfrac12(\lambda_1-\lambda_3)\sin 2\phi$, attained when the remaining mass $\sin^2\phi$ sits on the opposite extremal eigenvector. So Assumption~\ref{ass:F} is a quantitative alignment: the angle between $\xiVec$ and the nearest eigenvector is $O(|\log|x||^{-3})$ at the core. No condition is placed on the eigenvalue $\lambda_j$ carried by that eigenvector. The stretching $\alpha = \xiVec \cdot S\xiVec = \sum_i \lambda_i c_i^2$ may be maximal ($j=1$), and this is what distinguishes the present criterion from those built on $S\om$.

\subsection{The advection--curvature decomposition and the exact cases}

For the vorticity direction the unprojected strain has a transport form. From $\nabla\uVec = S + \Omega_{spin}$ with $\Omega_{spin}\mathbf v = \frac12 \om \times \mathbf v$ and $\om \times \xiVec = 0$, one has $S\xiVec = (\xiVec \cdot \nabla)\uVec$, and the identity $\nabla(\uVec\cdot\xiVec) = (\uVec\cdot\nabla)\xiVec + (\xiVec\cdot\nabla)\uVec + \uVec\times(\nabla\times\xiVec) + \xiVec\times(\nabla\times\uVec)$ with $\nabla\times\uVec \parallel \xiVec$ gives $S\xiVec = \nabla h - (\uVec\cdot\nabla)\xiVec - \uVec\times(\nabla\times\xiVec)$, where $h = \uVec\cdot\xiVec$ is the normalized helicity. Splitting $\uVec = h\xiVec + \uVec_\perp$ and using $\xiVec\times(\nabla\times\xiVec) = -(\xiVec\cdot\nabla)\xiVec$, the two longitudinal contributions cancel and
\begin{equation}\label{eq:unified_strain}
F_{tan}(\xiVec) = \underbrace{\nabla_{tan} h}_{\text{helicity gradient}} - \underbrace{(\uVec_\perp\cdot\nabla)\xiVec}_{\text{transverse advection}} - \underbrace{\mathbb{P}_\xi\big[\uVec_\perp\times(\nabla\times\xiVec)\big]}_{\text{transverse curvature}},
\qquad \nabla_{tan} h = \mathbb P_\xi \nabla h .
\end{equation}
(The transverse advection is kept on the right-hand side of \eqref{eq:pde_condensed}: moving it to the left would turn the drift into $2\uVec - h\xiVec$, whose longitudinal part has a singular divergence $\sim |x|^{-2}$ and would destroy the barrier structure of Section \ref{sec:barrier}.)

\medskip

Two exact configurations annihilate \eqref{eq:unified_strain}:
\begin{enumerate}[leftmargin=*]
\item \emph{Exact Beltrami.} If $\uVec = h\xiVec$ then $\uVec_\perp = 0$ and the two longitudinal terms cancel, so $S\xiVec = \nabla h$ and $F_{tan} = \nabla_{tan} h$. Thus $\xiVec$ is a strain eigenvector precisely when $\nabla h \parallel \xiVec$ (the normalized helicity is constant on the surfaces transverse to the vortex lines), with eigenvalue $\partial_{\xiVec} h$; the Beltrami depletion of the tangential strain is the alignment condition itself. (Geometric depletion of the nonlinearity in the near-Beltrami flows was considered in \cite{FarhatGrujic2018}.)
\item \emph{Exact axisymmetric no-swirl.} If $\xiVec = \hat{\boldsymbol\theta}$ and $\uVec = u_\rho\hat{\boldsymbol\rho} + u_z\hat{\mathbf z}$ at a point with cylindrical radius $\rho \ge R_0/2 > 0$, then $h = 0$, $(\uVec_\perp\cdot\nabla)\hat{\boldsymbol\theta} = 0$, and $\uVec_\perp\times(\nabla\times\hat{\boldsymbol\theta}) = -\frac{u_\rho}{\rho}\hat{\boldsymbol\theta}$ is killed by the projection; so $F_{tan} \equiv 0$. (These flows are globally regular \cite{Ladyzhenskaya1968, Ukhovskii1968}.)
\end{enumerate}

\section{Application: Strain-Aligned, Inflow-Controlled Cores}
\label{sec:application}

Theorem \ref{thm:main} needs four things from the flow near a critical spatial point singularity: the inflow condition (H$_k$); a sub-critical tangential strain; the cap condition on the boundary data of the core region; and a logarithmic modulus in time of the direction at the inner scale. In this section we (1) prove two Biot-Savart estimates at the core, (2) \emph{show that (H$_k$) follows from approximate symmetry of the vorticity magnitude about the local vortex axis}, and (3) package these, together with the alignment condition of Section \ref{sec:strain}, into a single application. 

\medskip

We write $\es$ for a fixed unit vector (the local vortex axis) and $B=B_{R_0}$.

\subsection{The endgame from Part I}

We recall the global analytical endgame established in the companion work \cite{GrujicProj1}, which specifies the functional threshold on the direction field that forces evasion of the singularity.
\begin{theorem}[Global Evasion via Logarithmic Depletion \cite{GrujicProj1}]\label{thm:project1_recall}
Let $\uVec$ be a strong solution to the 3D Navier-Stokes equations on $(0, T^*)$, with $T^*$ being the first potential singular time, and suppose the velocity field satisfies $\uVec(\cdot, T_0) \in L^\infty(\R^3)$ for some $T_0 < T^*$. If the flow exhibits a critical point singularity ($\omega \sim |x|^{-2} \in L^{3/2, \infty}$) satisfying the uniform-in-time bound $\omega \in L^\infty\bigl((T^*-\delta, T^*); L^{3/2, \infty}(\R^3)\bigr)$, and the vorticity direction maintains the uniform-in-time log-weighted $\bmo$ regularity $\xiVec \in L^\infty\bigl((T^*-\delta, T^*); \bmo_{1/|\log r|}(\R^3)\bigr)$, then the vortex stretching is structurally depleted. This forces the 1D geometric scale of sparseness of the intense advective region to scale as $r_s \sim \|\uVec\|_\infty^{-1} (\log \|\uVec\|_\infty)^{-1}$, falling below the uniform radius of spatial analyticity $\rho_{an} \sim \nu / \|\uVec\|_\infty$. The geometric overlap allows the harmonic measure maximum principle to avert the blow-up.
\end{theorem}

\subsection{Two Biot--Savart estimates at the core}

We use the elementary identity
\begin{equation}\label{eq:pi3}
\int_{\R^3}\frac{dy}{|y|^2\,|x-y|^2}=\frac{\pi^3}{|x|},
\end{equation}
(both sides are homogeneous of degree $-1$; the constant follows from $\widehat{|\cdot|^{-2}}=2\pi^2|k|^{-1}$ and
$\widehat{|\cdot|^{-1}}=4\pi|k|^{-2}$).

\medskip

Split $\uVec=\uVec_{near}+\uVec_{far}$, where $\uVec_{near}$ is the Biot--Savart velocity of $\omVec\mathbf 1_B$ and
$\uVec_{far}$ that of $\omVec\mathbf 1_{\R^3\setminus B}$:
\[
\uVec_{near}(x)=\frac{1}{4\pi}\int_B\frac{(x-y)\times\omVec(y)}{|x-y|^3}\,dy .
\]

\begin{assumption}\label{ass:far}
$C_{far}:=\sup_{t\in(t_0,T^*)}\ \sup_{|x|\le R_0/2}|\uVec_{far}(x,t)|<\infty$.
\end{assumption}
This is finite for instance if $\omVec\in L^\infty_t L^p(\R^3\setminus B)$ for some $p\in(3/2,3)$; for $|x|\le R_0/2$ the
kernel is bounded by $4R_0^{-2}$ on $B_{2R_0}\setminus B$ and by $|y|^{-2}$ beyond. It is a statement about the flow away
from the singular point, where the solution is smooth up to $T^*$.

\begin{lemma}[Biot--Savart at the core]\label{lem:BS}
Let $|\omVec|\le C_0|y|^{-2}$ on $B$ and let Assumption~\ref{ass:far} hold. Then for $|x|\le R_0/2$:
\begin{enumerate}[label=(\alph*),leftmargin=*]
\item\label{it:sup} $\displaystyle |\uVec(x)|\le\frac{\pi^2C_0}{4\,|x|}+C_{far}$; in particular $|x|\,|\uVec(x)|\le C_u:=\frac{\pi^2C_0}{4}+\frac{R_0C_{far}}{2}$.
\item\label{it:rad} Write $\omVec=\omega\,\es+\tilde\omVec$ on $B$ and let $\bar\omega$ be the average of $\omega$ over circles
centred on the axis $\R\es$ through the singular point (the azimuthal average about $\es$). Then
\begin{equation}\label{eq:radial}
\uVec_{near}(x)\cdot\hx=-\frac{1}{4\pi}\,(\es\times\hx)\cdot\nabla V_{\omega-\bar\omega}(x)
+\frac{1}{4\pi}\int_B\frac{\big((x-y)\times\tilde\omVec(y)\big)\cdot\hx}{|x-y|^3}\,dy,
\qquad V_f(x):=\int_B\frac{f(y)}{|x-y|}\,dy ,
\end{equation}
and consequently, if $|\omega-\bar\omega|\le\epsilon_\varphi|y|^{-2}$ and $|\tilde\omVec|\le\epsilon_d\,\omega$ on $B$,
\begin{equation}\label{eq:radial_bound}
|x|\,\big|\uVec(x)\cdot\hx\big|\ \le\ \frac{\pi^2}{4}\big(\epsilon_\varphi+C_0\,\epsilon_d\big)+\frac{R_0}{2}\,C_{far}.
\end{equation}
\end{enumerate}
\end{lemma}

\begin{proof}
\ref{it:sup} $|\uVec_{near}(x)|\le\frac{1}{4\pi}\int_B\frac{|\omVec(y)|}{|x-y|^2}dy\le\frac{C_0}{4\pi}\cdot\frac{\pi^3}{|x|}$ by \eqref{eq:pi3}.

\medskip

\ref{it:rad} The $\es$-part of $\uVec_{near}$ is $\frac{1}{4\pi}W\times\es$ with $W(x)=\int_B\frac{(x-y)\,\omega(y)}{|x-y|^3}dy=-\nabla V_\omega(x)$.
Hence $\frac{1}{4\pi}(W\times\es)\cdot\hx=-\frac{1}{4\pi}(\es\times\hx)\cdot\nabla V_\omega$. Now $(\es\times x)\cdot\nabla$ is the
generator of rotations about the axis $\R\es$; the ball $B$ and the kernel $|x-y|^{-1}$ are invariant under these rotations, so
$(\es\times x)\cdot\nabla V_f=V_{(\es\times y)\cdot\nabla f}$, which vanishes for $f=\bar\omega$. This gives
\eqref{eq:radial}. For \eqref{eq:radial_bound}: $|(\es\times\hx)\cdot\nabla V_{\omega-\bar\omega}|\le|\nabla V_{\omega-\bar\omega}|
\le\int_B\frac{|\omega-\bar\omega|}{|x-y|^2}dy\le\frac{\pi^3\epsilon_\varphi}{|x|}$, the $\tilde\omVec$-term is at most
$\frac{1}{4\pi}\int_B\frac{|\tilde\omVec|}{|x-y|^2}dy\le\frac{\pi^2C_0\epsilon_d}{4|x|}$, and $|\uVec_{far}\cdot\hx|\le C_{far}$.
\end{proof}

\begin{remark}
For $\omega$ exactly symmetric about $\R\es$ and $\tilde\omVec=0$, \eqref{eq:radial} gives $\uVec_{near}\cdot\hx\equiv0$: the
near-field velocity is purely azimuthal about the vortex axis. In the spherical case $\omVec=\Phi|y|^{-2}\es$ with
$\Phi$ constant one has explicitly $\uVec_{near}=\Phi\,\hx\times\es/|x|$ (Newton's theorem: $V_\omega'(r)=-4\pi\Phi/r$).
The radial velocity at the core is produced only by the asymmetry of the magnitude about the vortex axis and by the
deviation of the direction from it.
\end{remark}

\subsection{(H$_k$) from local axial symmetry}

\begin{proposition}\label{prop:Hk}
Let $\omega=\Phi|x|^{-2}$ on $A$ with $0<\Phi\le C_0$ and $|\nabla\ln\Phi|\le C_\Phi|x|^{-1}$, and suppose on $B\times(t_0,T^*)$
\[
|\omega-\bar\omega|\le\epsilon_\varphi|y|^{-2},\qquad |\xiVec-\es|\le\epsilon_d ,
\]
with Assumption~\ref{ass:far}. Then Assumption~\ref{ass:PH} (profile, $|\uVec|\le C_u|x|^{-1}$, and \textup{(H$_k$)}) holds on
$A\cap\{|x|\le R_0/2\}$ for every $k\in(0,3)$ with
\begin{equation}\label{eq:Hk_verified}
\frac{\pi^2}{4}\big(\epsilon_\varphi+C_0\,\epsilon_d\big)+\frac{R_0}{2}\,C_{far}+2\nu C_\Phi\ \le\ (3-k)\,\nu .
\end{equation}
\end{proposition}

\begin{proof}
(H$_k$) asks for $|x|(u_r)_-+2\nu C_\Phi\le(3-k)\nu$; use \eqref{eq:radial_bound} with $\tilde\omVec=\omega(\xiVec-\es)$.
The bound on $|\uVec|$ is Lemma~\ref{lem:BS}\ref{it:sup}. Replace $R_0$ by $R_0/2$ in the definition of $A$.
\end{proof}

\begin{remark}
\begin{enumerate}[leftmargin=*]
\item $\epsilon_\varphi$, $C_0\epsilon_d$ and $\nu$ all carry the dimension of a diffusivity; \eqref{eq:Hk_verified} is a
core Reynolds-number condition on the \emph{asymmetric part} of the core only. The symmetric part, however large, does not enter.
\item The hypothesis $|\xiVec-\es|\le\epsilon_d$ is an interior hypothesis, but a scale-free one (a fixed small angle,
with no decay as $|x|\to0$). Theorem~\ref{thm:main} converts it into the decaying modulus $|\xiVec-e_t|\lesssim1/|\log|x||$, so there is no
circularity: the input is weaker than the output. It also implies the cap condition of Theorem~\ref{thm:main} with $\cos\phi_0=1-\epsilon_d^2/2$.
\end{enumerate}
\end{remark}

\subsection{The application}

With \eqref{eq:Ftan} in hand, Assumption~\ref{ass:F} is a quantitative alignment of the vorticity direction with an eigenvector of $S$ at the core, and we can state the application.

\begin{definition}[strain-aligned, inflow-controlled core]\label{def:core}
A critical spatial point singularity at $(0,T^*)$ with inner scale $\ell(t)$ has a \emph{strain-aligned, inflow-controlled core}
if there are $R_0$, $\es$, $k\in(0,3)$ and constants such that on $A$:
\begin{enumerate}[label=(\roman*),leftmargin=*]
\item (profile) $\omega=\Phi|x|^{-2}$, $0<\Phi\le C_0$, $|\nabla\ln\Phi|\le C_\Phi|x|^{-1}$, and $|\nabla\xiVec|\le C_1/\max(|x|,\ell(t))$;
\item (inflow) \eqref{eq:Hk_verified} holds with the asymmetry constants $\epsilon_\varphi$, $\epsilon_d$ of Proposition~\ref{prop:Hk}
and Assumption~\ref{ass:far} --- or, more generally, (H$_k$) holds directly;
\item (alignment) $|P_{\xiVec^\perp}S\xiVec|\le\Lambda|x|^{-2}|\log|x||^{-3}$, with $R_0\le e^{-6/k}$ and $\frac{\Lambda}{\nu k}|\log R_0|^{-2}<\frac14\epsilon_d^2$;
\item (core history) $\omega_{core}(\sigma;t)\le\Lambda_0/|\log\sigma|$ for $\ell(t)^2/\nu \le \sigma\le\sigma_1$, $t\in(t_0+\sigma_1,T^*)$, where $\omega_{core}$ is
the oscillation of $\xiVec$ over the inner balls $\{|y|\le\ell(s)\}$, $s\in[t-\sigma,t]$.
\end{enumerate}
\end{definition}

\begin{remark}
The core history condition as stated above is somewhat generous. As an illustration, consider the case of asymptotically self-similar profiles and
let $\tau$ denote the (logarithmically) rescaled time. Then the condition is equivalent to the the following: the direction on the inner ball locks 
onto a limit direction at rate of at least $1/\tau$.
If the rescaled solution converges to a profile along a stable manifold, the convergence is exponential in $\tau$, far more than required.
This applies equally to \emph{unstable profiles} such as those approached along a finite-codimension stable manifold.
The marginal case, algebraic $1/\tau$ decay along a neutral (center-manifold) mode, is exactly the threshold.
\end{remark}

\begin{theorem}[Application]\label{thm:app}
If a critical spatial point singularity has a strain-aligned, inflow-controlled core, then there are $\rho_1>0$ and $C<\infty$,
depending only on the constants in Definition~\ref{def:core} and on $\nu$, such that
\[
\sup_{t\in(t_0+\sigma_1,T^*)}\ \|\xiVec(\cdot,t)\|_{\bmo_{1/|\log r|}(B_{\rho_1})}\le C .
\]
Consequently, by the criterion of Part~I, $(0,T^*)$ is not a singular point.
\end{theorem}

\begin{proof}
Definition~\ref{def:core}(i)--(ii) and Proposition~\ref{prop:Hk} give Assumptions~\ref{ass:PH} and~\ref{ass:grad}; (ii) with
$|\xiVec-\es|\le\epsilon_d$ gives the cap condition with $\delta_0=\cos2\phi_0$, $\cos\phi_0=1-\epsilon_d^2/2$, and the smallness
condition on $R_0$ in Theorem~\ref{thm:main} reads $\frac{\Lambda}{\nu k}|\log R_0|^{-2}<\frac12(1-\delta_0)$, which is implied by (iii) since
$1-\delta_0=2\sin^2\phi_0\ge\epsilon_d^2/2$; (iii) is Assumption~\ref{ass:F} by \eqref{eq:Ftan}; (iv) is the core-history condition of Corollary~\ref{cor:bmo}.
Apply Corollary~\ref{cor:bmo}.
\end{proof}

\section{Obstructions}\label{sec:obstructions}

At the end, we indicate why point-concentrated tube cores and Beltrami cores may not be good candidates for a base state
(in contrast to the configuration of exact alignment). This is another indication that the 
geometric-analytic architecture presented here and in the companion paper \cite{GrujicProj1} is complementary
to the classical geometric regularity theory.

\subsubsection*{Point-concentrated tube cores generate critical self-strain}

\begin{proposition}\label{prop:selfstrain}
Let $\omVec=\Phi(\vartheta)\,|x|^{-2}\,\es$ on $\R^3$, with $\Phi$ bounded, positive and depending only on the polar angle
$\vartheta$ from $\es$ (an $\es$-axisymmetric, scale-invariant tube core), and let $\uVec$ be its Biot--Savart velocity.
Then $F_{tan}=P_{\es^\perp}S\es=\partial_{\es}\uVec$ is homogeneous of degree $-2$ and does not vanish identically.
In the spherical case $\Phi\equiv\Phi_0$,
\begin{equation}\label{eq:selfstrain}
S\es=\partial_{\es}\uVec=-\frac{2\Phi_0\cos\vartheta}{|x|^2}\ \hx\times\es,\qquad
|F_{tan}|=\frac{\Phi_0}{|x|^2}\,|\sin2\vartheta| ,\qquad \es\cdot S\es=0 .
\end{equation}
\end{proposition}

\begin{proof}
For $\omVec=\omega\,\es$, $\uVec=\frac{1}{4\pi}W\times\es$ with $W=-\nabla V_\omega$, and $S\es=(\es\cdot\nabla)\uVec-\frac12\omVec\times\es=\partial_{\es}\uVec$.
So $F_{tan}=\partial_{\es}\uVec=-\frac{1}{4\pi}\nabla(\partial_{\es}V_\omega)\times\es$, which is tangential (orthogonal to $\es$)
and vanishes iff $\nabla(\partial_{\es}V_\omega)\parallel\es$, i.e.\ iff $\partial_{\es}V_\omega$ depends on $z=x\cdot\es$ alone.
Write $V_\omega=-4\pi\langle\Phi\rangle\log|x|+v(\vartheta)$ with $\langle\Phi\rangle$ the spherical mean and $\Delta_{S^2}v=-4\pi(\Phi-\langle\Phi\rangle)$;
then $\partial_{\es}V_\omega=|x|^{-1}w(\vartheta)$ is homogeneous of degree $-1$, and a function of $z$ alone with this homogeneity is $a/z$,
i.e.\ $w(\vartheta)=a/\cos\vartheta$. Boundedness of $\Phi$ makes $w$ bounded on the sphere, forcing $a=0$, hence $\partial_{\es}V_\omega\equiv0$;
but then $\Delta V_\omega=-4\pi\Phi(\vartheta)|x|^{-2}$ would be independent of $z$, which is impossible for bounded positive $\Phi$.
For \eqref{eq:selfstrain}: Newton's theorem gives $V'_\omega=-4\pi\Phi_0/r$, so $\uVec=\Phi_0\,\hx\times\es/|x|=\Phi_0\,x\times\es/|x|^2$, and
$\partial_{\es}(x/|x|^2)=\es/|x|^2-2(x\cdot\es)x/|x|^4$; crossing with $\es$ kills the first term.
\end{proof}

\begin{remark}
The strain in \eqref{eq:selfstrain} is pure tilting at the critical rate $\Phi_0/|x|^2$, to be compared with the diffusive rate
$\nu/|x|^2$ of $\Lop$. In the $\theta$-equation of Theorem~\ref{thm:main} the source is $F_{tan}\cdot e=F_{tan}\cdot(e-\xiVec)$, and a critical
source cannot be absorbed by any barrier: for the radial model with $|x|b_r=(k+1)\nu$, the equation $\Lop\psi=\Lambda|x|^{-2}$ with
$\psi=0$ at $|x|=\ell$ has the solution $\psi=\frac{\Lambda}{k\nu}\log(|x|/\ell)$, so the direction winds logarithmically as
$\ell(t)\to0$. Exact axisymmetry avoids this only because the vorticity of a ring is constant along the tube; point concentration
breaks the depletion at order one, and Proposition~\ref{prop:selfstrain} shows no choice of $\es$-axisymmetric shape factor restores it.
\end{remark}

\subsubsection*{Beltrami cores have critical twist}

\begin{proposition}\label{prop:twist}
Suppose $\uVec=h\,\xiVec+\uVec_\perp$ on $A$ with $|\nabla\times\uVec_\perp|\le\tfrac12\omega$, $|h|\le|\uVec|\le C_u|x|^{-1}$, and suppose
in addition that the core is two-sidedly critical, $\omega\ge c_0|x|^{-2}$, on $A$ (a hypothesis of this proposition only; cf.\ Remark~\ref{rem:flux}).
Then
\[
\big|\xiVec\cdot(\nabla\times\xiVec)\big|\ \ge\ \frac{c_0}{2C_u}\,\frac{1}{|x|},
\qquad\text{hence}\qquad
|x|\,|\nabla\xiVec|\ \ge\ \frac{c_0}{2\sqrt2\,C_u} .
\]
\end{proposition}

\begin{proof}
$\omega\xiVec=\nabla\times\uVec=\nabla h\times\xiVec+h\,\nabla\times\xiVec+\nabla\times\uVec_\perp$. Dot with $\xiVec$:
$h\,\xiVec\cdot(\nabla\times\xiVec)=\omega-\xiVec\cdot(\nabla\times\uVec_\perp)\ge\omega/2\ge\frac{c_0}{2|x|^2}$.
Divide by $|h|\le C_u|x|^{-1}$ and use $|\nabla\times\xiVec|\le\sqrt2|\nabla\xiVec|$.
\end{proof}

\begin{remark}
The twist $\xiVec\cdot\nabla\times\xiVec$ is the rate at which the direction rotates about itself along vortex lines. A near-Beltrami
core with $|\uVec|\sim|x|^{-1}$ and $\omega\sim|x|^{-2}$ has twist $\sim|x|^{-1}$: the vortex lines are helices whose pitch is comparable
to the distance from the singular point. By Stokes' theorem on a disc of radius $\rho\sim|x_0|$ normal to the mean direction, the
circulation of $\xiVec-e$ around the boundary is $\gtrsim(c_0/C_u)\rho$, so $|\xiVec-e|\gtrsim c_0/C_u$ on a set of positive proportion of
the disc: the mean oscillation of $\xiVec$ on balls at distance $\sim r$ from the core is bounded below by a fixed constant at every scale,
and $\|\xiVec\|_{\bmo_{1/|\log r|}}\gtrsim|\log r|$. Neither the core-history condition (iv) nor the cap condition can hold for such a
core. (For a Beltrami core the two-sided bound is the natural hypothesis, since the helical vortex lines wind around the axis and are not confined to a hemisphere; Proposition~\ref{prop:flux} is not in conflict with it.)
\end{remark}

\medskip

\section{Acknowledgments}
The work is supported in part by the National Science Foundation grant DMS 2307657.

\end{document}